\documentclass{article}
\usepackage[a4paper]{geometry}
\usepackage{amsmath}
\usepackage{amssymb}
\usepackage{amsthm}
\usepackage{mathrsfs}
\usepackage{mathtools}

\usepackage{tikz}
\usetikzlibrary{calc} 
\usepackage{hyperref}
\usepackage{cleveref}
\usepackage{tikz-cd}
\usepackage{mathtools}

\usepackage{placeins}

\usepackage{float}

\usepackage{comment}

\usepackage{enumerate}

\newtheorem{thm}{Theorem}[section]
\crefname{thm}{Theorem}{Theorems}
\newtheorem{cor}[thm]{Corollary}
\newtheorem{prop}[thm]{Proposition}
\crefname{prop}{Proposition}{Propositions}
\newtheorem{lem}[thm]{Lemma}
\crefname{lem}{Lemma}{Lemmas}
\newtheorem{clm}[thm]{Claim}

\newtheorem{defn}[thm]{Definition}

\crefname{defn}{Definition}{Definitions}

\newtheorem*{ack*}{Acknowledgements}

\numberwithin{equation}{section}

\newcommand{\PP}{\mathbb{P}}

\newcommand{\co}{\operatorname{co}}

\usepackage{titling}
\title{Sharp stability
of the Brunn-Minkowski inequality\\
via optimal mass transportation}
\author{Alessio Figalli
\thanks{Department of Mathematics, ETH Zurich, Switzerland. email: alessio.figalli@math.ethz.ch},
Peter van Hintum
\thanks{New College, University of Oxford, UK. email: peter.vanhintum@new.ox.ac.uk (corresponding author)}, 
Marius Tiba
\thanks{Mathematical Institute, University of Oxford, UK. email: marius.tiba@maths.ox.ac.uk}}

\begin{document}

\maketitle

\begin{abstract}
The Brunn-Minkowski inequality, applicable to bounded measurable sets $A$ and $B$ in $\mathbb{R}^d$, states that $|A+B|^{1/d} \geq |A|^{1/d}+|B|^{1/d}$. Equality is achieved if and only if $A$ and $B$ are convex and homothetic sets in $\mathbb{R}^d$.
The concept of stability in this context concerns how, when approaching equality, sets $A$ and $B$ are close to homothetic convex sets. In a recent breakthrough \cite{BMStab}, the authors of this paper proved the following folklore conjectures on the sharp stability  for the Brunn-Minkowski inequality:\\
(1) A linear stability result concerning the distance from $A$ and $B$ to their respective convex hulls.\\
%(1) Sets $A$ and $B$ exhibit linear proximity (measured by the deficit) to their respective convex hulls.\\
(2) A quadratic stability result concerning the distance from $A$ and $B$ to their common convex hull. \\
As announced in \cite{BMStab}, in the present paper, we leverage (1) in conjunction with a novel optimal transportation approach to offer an alternative proof for (2).  
\end{abstract}

\setcounter{tocdepth}{2}
%\tableofcontents

\section{Introduction}

Given measurable sets $X,Y\subset \mathbb{R}^n$ with positive measure, the Brunn-Minkowski inequality says
$$|X+Y|^{\frac{1}{n}} \ge |X|^{\frac{1}{n}}+|Y|^{\frac{1}{n}}.$$ 
More naturally, for equal sized measurable sets $A,B\subset \mathbb{R}^n$ and a parameter $t\in(0,1)$ this is equivalent to
$$|tA+(1-t)B|\geq |A|,$$
with equality for equal convex sets $A$ and $B$ (less a measure zero set). Here, $A+B=\{a+b\mid a\in A,\text{ and }b\in B\}$ is the \emph{Minkowski sum}, $tA:=\{ta: a\in A\}$, and $|\cdot|$ refers to the outer Lebesgue measure. The Brunn-Minkowski inequality is a fundamental tool in analysis and geometry going back to the 19th century, the importance of which is expertly documented in \cite{gardner2002brunn}.

The Brunn-Minkowski inequality is part of a vast body of geometric inequalities, such as the isoperimetric inequality, the Sobolev inequality, the Pr\'ekopa-Leindler inequality, and the Borell-Brascamb-Lieb inequality (e.g. Figure 1 in \cite{gardner2002brunn}). The famous isoperimetric inequality states that, for a given volume, the body minimizing its perimeter is the ball. The isoperimetric inequality follows from Brunn-Minkowski by taking $A$ a ball and letting $t$ tend to zero. The Pr\'ekopa-Leindler inequality asserts that for $t\in(0,1)$ and functions $f,g,h\colon \mathbb{R}^n\to\mathbb{R}_{\geq 0}$ with the property that $h(tx+(1-t)y)\geq f^{t}(x)g^{1-t}(y)$ for all $x,y\in\mathbb{R}^n$ and $\int f=\int g$, we have $\int h\geq \int f$ with equality if and only if $f(x)=ag(x-x_0)$ is a log-concave function for some $a\in\mathbb{R}_{>0}$ and $x_0\in\mathbb{R}^n$. Pr\'ekopa-Leindler implies Brunn-Minkowski by taking $f$ and $g$ to be the indicator functions of $A$ and $B$. The Pr\'ekopa-Leindler inequality in turn is subsumed by the Borell-Brascamb-Lieb inequality. Studying these inequalities and their stabilities has sparked a fruitful field of research in recent years.

The stability of  Brunn-Minkowski asks for the structure of sets $A$ and $B$ which are close to attaining equality in Brunn-Minkowski. This study goes back to the work of for instance Diskant \cite{diskant1973stability} and Ruzsa \cite{ruzsa1997brunn}.  Two folklore conjectures concern the stability of  Brunn-Minkowski:  if we are within a factor of $1+\delta$ from equality, then the distance from the sets $A$ and $B$ to a common convex set is $O_{d,t}(\sqrt{\delta})$, and furthermore, the distance from to their individual convex hulls is $O_{d,t}(\delta)$. These conjectures have received a lot of attention becoming central problems in analysis and convex geometry (see e.g. \cite{Figalli09,Figalli10amass,christ2012planar,christ2012near,eldan2014dimensionality,figalli2015quantitative,figalli2015stability,figalli2017quantitative,Barchiesi,carlen2017stability,figalli2021quantitative,van2021sharp,van2020sharp,van2023locality,SharpDelta,planarBM,BMStab}). Recently, the present authors resolved these conjectures in \cite{BMStab} (stated as \Cref{main_thm_1} and \Cref{LinearThmGeneral} below).
\newpage

The stability of the isoperimetric inequality was first explored in 1921 by Bonnesen \cite{bonnesen1921amelioration} who settled the planar case. The optimal result in higher dimensions was established only in 2008 by Fusco, Maggi, and Pratelli \cite{FMP08}.  In a cornerstone paper, Figalli, Maggi, and Pratelli \cite{Figalli09,Figalli10amass} used mass transportation techniques to generalize this to a sharp stability of the anisotropic isoperimetric inequality while simultaneously proving the following sharp stability for the Brunn-Minkowski inequality for convex sets.

%A first step towards the folklore conjectures was taken by Figalli, Maggi, and Pratelli \cite{Figalli09,Figalli10amass} who used mass transportation techniques to establish the following cornerstone result for convex sets.

%This area started with the cornerstone result of Figalli, Maggi, and Pratelli \cite{Figalli09,Figalli10amass} who used mass transportation techniques to prove the following result for convex sets.

\begin{thm}[Figalli, Maggi, and Pratelli \cite{Figalli09,Figalli10amass}]
\label{FMP}
For all $n \in \mathbb{N}$ and $t \in (0,1/2]$, there are computable constants $ c^{\ref{FMP}}_n>0$ such that the following holds. Assume that $A,B\subset\mathbb{R}^n$,  are convex sets with equal volume so that 
$$ \left|tA+(1-t)B\right| \leq (1+\delta)|A|.$$
Then, up to translation\footnote{That is, there exists $x\in\mathbb{R}^n$ so that $|(A+x) \triangle B|\leq c^{\ref{FMP}}_{n}\sqrt{\frac{\delta}{t}}|A|$.}, $$|A \triangle B| \leq c^{\ref{FMP}}_{n}\sqrt{\frac{\delta}{t}}|A|.$$
\end{thm}

The aim of this paper is to develop a different mass transportation approach on the stability of the Brunn-Minkowski problem in order to   strengthen  the above result to non-convex sets.
\begin{thm}
\label{main_thm_5}
For all $n \in \mathbb{N}$ and $t \in (0,1/2]$, there are computable constants $ c^{\ref{main_thm_5}}_n, d^{\ref{main_thm_5}}_{n,t}, g^{\ref{main_thm_5}}_{n,t}>0$ such that the following holds. Assume $\delta\in[0, d^{\ref{main_thm_5}}_{n,t})$,   $\gamma\in[0, g^{\ref{main_thm_5}}_{n,t})$,  and assume that $A,B\subset\mathbb{R}^n$,  are measurable sets with equal volume so that 
$$ \left|tA+(1-t)B\right| \leq (1+\delta)|A|\qquad \text{ and }\qquad|\co(A)\setminus A|+|\co(B)\setminus B|\leq \gamma|A|.$$
Then, up to translation, $$|A \triangle B| \leq c^{\ref{main_thm_5}}_{n}\sqrt{\frac{\delta+\gamma}{t}}|A|.$$
\end{thm}

In recent work of the current authors \cite{BMStab}, the following linear stability result to the convex hull of $A$ and $B$ was established, solving one of the aforementioned conjectures.

\begin{thm}[\cite{BMStab}]\label{LinearThmGeneral}
For $n\in\mathbb{N}$ and $t\in(0, 1/2]$, there are constants $c_{n,t}^{\ref{LinearThmGeneral}},d^{\ref{LinearThmGeneral}}_{n,t}>0$ such that the following holds. Assume $\delta\in[0,d^{\ref{LinearThmGeneral}}_{n,t})$, and assume $A,B\subset\mathbb{R}^n$ are measurable sets of equal volume so that $|tA+(1-t)B|\leq (1+\delta)|A|$, then
$$|\co(A)\setminus A|+|\co(B)\setminus B|\leq c_{n,t}^{\ref{LinearThmGeneral}}\delta |A|.$$ 
\end{thm}
A notable application of  \Cref{main_thm_5} is that, in combination with \Cref{LinearThmGeneral}, it gives the following result. 

\begin{cor}\label{badtcor}
For all $n \in \mathbb{N}$ and $t \in (0,1/2]$, there are computable constants $ c_{n,t}^{\ref{badtcor}}, d^{\ref{badtcor}}_{n,t}>0$ such that the following holds. Assume $\delta\in[0, d^{\ref{badtcor}}_{n,t})$ and assume that $A,B\subset\mathbb{R}^n$,  are measurable sets with equal volume so that  $\left|tA+(1-t)B\right| \leq (1+\delta)|A|.$ Then, up to translation, $$|A \triangle B| \leq c^{\ref{badtcor}}_{n,t}\sqrt{\delta}|A|.$$   
\end{cor}
This corollary is a weaker instance of the following quadratic stability recently proved by the current authors. 

\begin{thm}[\cite{BMStab}]
\label{main_thm_1}
For all $n \in \mathbb{N}, n \geq 2$ and $ t\in (0,1/2]$, there are computable constants $c_n^{\ref{main_thm_1}},d^{\ref{main_thm_1}}_{n,t}>0$ such that the following holds. Assume $\delta \in [0, d^{\ref{main_thm_1}}_{n,t})$ and let $A,B\subset \mathbb{R}^n$ be measurable sets with equal volume satisfying
$$ |tA+(1-t)B| =(1+\delta)|A|.$$
Then, up to translation\footnote{That is, there exist $x,y\in\mathbb{R}^n$ so that $x+A,y+B\subset K$ and $|K\setminus (x+A)|+|K\setminus (y+B)|\leq t^{-c^{\ref{main_thm_1}}n^8}\delta^{\frac{1}{2}}|A|$.}, there is a convex set $K\supset A \cup B$ such that
$$|K\setminus A|=|K\setminus B| \le c_n^{\ref{main_thm_1}}\sqrt{\frac{\delta}{t}}|A|.$$
\end{thm}

Note that $\frac{|\co(A\cup B)\setminus A| }{|A\triangle B|} \geq 1/2$, but a priori we don't have any lower bound in terms of $n$.   However, as a consequence of \cite[Theorem 1.7]{BMStab} these two measures are actually equivalent for near-convex sets $A,B$, i.e., with $|\co(A)\setminus A|+|\co(B)\setminus B|=O_{n,t}(\delta)|A|$. Hence, the main difference between \Cref{main_thm_1} and \Cref{badtcor} is in the $t$ dependence of the stability constant. Actually, even combining \Cref{main_thm_5} with the optimal version of  \Cref{LinearThmGeneral} (see Conjecture 14.1 in \cite{BMStab}) would still not obtain the optimal $t$ dependence provided by \Cref{main_thm_1}.

\newpage

The first contribution to the study of sumset stability was made by Freiman \cite{freiman1959addition} in dimension $n=1$.  Freiman's celebrated $3k-4$ Theorem \cite{freiman1959addition,lev1995addition,stanchescu1996addition} from additive combinatorics,
implies a strong version of \Cref{LinearThmGeneral} in dimension $1$. If $t \in (0,1/2]$ and $A,B \subset \mathbb{R}$ are measurable sets with equal volume such that $|tA+(1-t)B| \leq (1+\delta)|A|$ with $\delta <t$, then $|\co(A)\setminus A| \leq t^{-1}\delta |A|$ and $|\co(B)\setminus B| \leq (1-t)^{-1}\delta|B|$, which is optimal. 

Stability in higher dimensions is considerably more difficult; in \cite{christ2012planar,christ2012near} Christ showed a qualitative result: if $n\in \mathbb{N}$, $t, \varepsilon\in (0,1/2]$ and $A,B \subset \mathbb{R}^n$ are measurable sets with equal volume such that $|tA+(1-t)B| \leq (1+\delta)|A|$ with $\delta$ sufficiently small in terms of $t ,n, \varepsilon$, then there exists a convex set $K$ such that, up to translation, $K \supset A,B$ and $|K\setminus A|=|K\setminus B| \leq \varepsilon|A|$. In a cornerstone result, Figalli and Jerison \cite{figalli2017quantitative} obtained the first quantitative bounds: $|K\setminus A|=|K\setminus B| \leq \delta^{(t/|\log(t)|)^{\exp(O(n))}}|A|$. A similar result for the  Pr\'ekopa-Leindler inequality was recently established by B\"or\"ocky, Figalli, and Ramos \cite{boroczky2022quantitative}.

Until recently, the only instance of \Cref{main_thm_1} for arbitrary sets was known in two dimensions due to van Hintum, Spink, and Tiba \cite{planarBM}. In an independent direction, van Hintum and Keevash \cite{SharpDelta} determined the optimal value $d_{n,t}=t^n$ for all values $n\in\mathbb{N}$ and $t\in(0,1/2]$ with the same bound on the distance to a common convex set as in the result of Figalli and Jerison.

Even partial results towards \Cref{main_thm_1} for restricted classes of sets $A$ and $B$ have received much attention. Recall that Figalli, Maggi, and Pratelli \cite{Figalli09, Figalli10amass} dealt with the case when $A$ and $B$ are convex. Figalli, Maggi, and Mooney \cite{Euclidean} settled the case when $A$ is a ball and $B$ is arbitrary. Barchiesi and Julin \cite{Barchiesi} extended the previous results to $A$ convex and $B$ arbitrary. Despite all these results supporting \Cref{main_thm_1}, a conclusive proof remained wide open and outside the scope of the available techniques for a long time.

%These papers have focused on controlling the weaker ``asymmetry'' measure $\inf_x|A\triangle (B+x)|$, which a priori does not control $|K\setminus A|$. In \cite{Figalli09, Figalli10amass}, Figalli, Maggi, and Pratelli established that if $n\in \mathbb{N}$, $t\in (0,1/2]$ and $A,B \subset \mathbb{R}^n$ are convex sets with equal volume such that $|tA+(1-t)B| \leq (1+\delta)|A|$ with $\delta$ sufficiently small in terms of $t$ and $n$, then, up to translation, $|A\triangle B| \leq O_d(t^{-1/2}\delta^{1/2})|A|$. Figalli, Maggi, and Mooney \cite{Euclidean} showed the analogous result when $A$ is a ball and $B$ is arbitrary. Note that this is closely related to the stability of the isoperimetric inequality. Barchiesi and Julin \cite{Barchiesi} extended the previous results to $A$ convex and $B$ arbitrary. Despite all these results supporting \Cref{main_thm_1}, a conclusive proof remained wide open and outside the scope of the available techniques for a long time.

The particular case of equal sets $A=B$ in \Cref{LinearThmGeneral} has been thoroughly investigated. Indeed, after establishing in \cite{figalli2015quantitative} some quantitative bounds for \Cref{LinearThmGeneral} for $A=B$ in all dimensions, Figalli and Jerison \cite{figalli2021quantitative} proved \Cref{LinearThmGeneral} for $A=B$ in dimensions $n=1,2,3$. Van Hintum, Spink, and Tiba  \cite{van2021sharp} proved \Cref{LinearThmGeneral} for $A=B$ in all dimensions. Moreover, they determined the optimal dependency on $t$. Furthermore, van Hintum, Spink, and Tiba  \cite[Theorem 1.1]{van2020sharp} established the optimal dependency on $d$ in dimensions $d \leq 4$ when $A=B$ is a hypograph of a function over a convex domain. Another closely related result by van Hintum and Keevash \cite{van2023locality} is that if $A \subset \mathbb{R}^n$ with $|\frac{A+A}{2}|\leq (1+\delta)|A|$ with $\delta<1$, then there exists an $A'\subset A$ with $|A'|\geq (1-\delta)|A|$ and $|\co(A')|=O_{n,1-\delta}(|A'|)$. 

For distinct sets $A$ and $B$, showing \Cref{LinearThmGeneral} has proved much more difficult. Van Hintum, Spink, and Tiba  in \cite[Theorem 1.5]{van2020sharp}, proved \Cref{LinearThmGeneral}, when $A$ and $B$ are hypograph of functions over the same convex domain. The only instance of \Cref{LinearThmGeneral} for arbitrary sets was established by van Hintum, Spink and Tiba \cite[Section 12]{planarBM} in two dimensions. In spite of these determined efforts, for arbitrary sets in higher dimensions a proof of \Cref{LinearThmGeneral} was only recently found by the present authors in \cite{BMStab}.

\begin{ack*}AF acknowledges the support of the ERC Grant No.721675 ``Regularity and Stability in Partial Differential Equations (RSPDE)'' and of the Lagrange Mathematics and Computation Research Center.
\end{ack*}

\subsection{Notation and conventions.}

Before starting our proofs, it is convenient to briefly explain the notation that we will use throughout the paper.
With $c>0$, we shall denote a universal constant independent of the dimension, while $c_n>0$ (and analogous notations) denote dimensional constants. Saying that the quantity $a$ is controlled by $O_n(b)$ means that $|a|\leq c_nb$, while notation $a=\Omega_n(b)$ means that $a\geq c_n|b|$. When a constant also depends on $t$, we write $c_{n,t}$.
To distinguish the constants that appear in the different statements, $c^{\ell.m}$ means that the constant $c$ is the one of Theorem $\ell.m$. 

Throughout the paper, we fix $n \in \mathbb{N}$ and either $t \in (0,1/2]$. We use $|\cdot|$ to denote the outer Lebesgue measure in $\mathbb{R}^n$.

Given $s \in \mathbb{R}$ and sets $X$ and $Y$ in $\mathbb{R}^n$, we define $sX=\{sx \colon x \in X\}$ and $X+Y=\{x+y \colon x\in X, y\in Y\}$. A set $X$ in $\mathbb{R}^n$ is convex if for all $t \in [0,1]$ we have $t X+ (1-t)X \subset X$.
The convex hull $\co(X)$ of a set $X$ in $\mathbb{R}^n$ is the intersection of all convex sets containing $X$. In particular, $\co(X)$ is a convex set. Two sets $X$ and $Y$ of $\mathbb{R}^n$ are homothetic if there exist a point $z$ in $\mathbb{R}^n$ and a scalar $s> 0$ such that $X=sY+z$.

Given a bounded convex set $X$ in $\mathbb{R}^n$, we define $\overline{X}$ as the closure of $X$, which is also a convex set. The vertices of $X$, denoted by $V(X)$, represent the set $V(X)=\{x \in \overline{X}  \colon \co(\overline{X} \setminus \{x\})\neq \co(\overline{X})\}$. It follows that $\overline{X}=\co(V(X))$. 

Measureable sets $X_1, \dots, X_k$ in $\mathbb{R}^n$ are said to form an essential partition of $\mathbb{R}^n$ if $|\cap_i X_i^c|=0$ and for $j_1 \neq j_2$, we have $|X_{j_1}\cap X_{j_2}|=0$. 
By a basis $e_1, \dots, e_n$ in $\mathbb{R}^n$, we mean an orthogonal set of vectors with unit length.
In light of \Cref{transportinitialreduction}, we can assume that the sets $A$ and $B$ (as well as all parts into which we subdivide $A$ and $B$) are compact.

\subsection{Overview of the proof}

We now turn to \Cref{main_thm_5}. The starting point is the optimal transport approach used in \cite{Figalli09} to prove a sharp stability result for the Brunn-Minkowski inequality on convex sets. In our case, the sets $A$ and $B$ are only $L^1$-close to being convex, and we want to obtain a final estimate where the gap in volume (i.e., $\gamma$) appears in the stability estimate with the same power as $\delta$. Because the optimal transport between arbitrary sets can behave very badly in terms of regularity, we consider the optimal transport map sending $\co(A)$ to $\co(B)$. This makes the first part of our argument (the first three steps in the outline below) very similar to the one in \cite{Figalli09}, but then we immediately face a series of challenges. The key issue is that the optimal transport proof of Brunn-Minkowski provides a control on the transport map only inside the set $A$ (although this map is defined in the whole convex hull), while for us it is crucial to obtain some bounds also in the remaining region $\co(A)\setminus A$. By a series of delicate arguments exploiting the monotonicity of the optimal map (we recall that the optimal map is the gradient of a convex function) and some interior regularity estimates, we obtain a radial control on the transport map along all rays emanating from the origin and contained inside $\co(A)$. This estimate by itself would be too weak. Still, the key observation is that we can repeat our argument by replacing the origin with an arbitrary point $o'$ inside $(1-\varepsilon)C_A$, and replacing our sets $A$ and $B$ with new sets $Q(A)$ and $Q(B)$, where $Q$ varies among all affine transformations with $||Q||_{op}, ||Q^{-1}||_{op}\leq \theta$ for some fixed large constant $\theta$.
Averaging our radial bound over $o'$ and $Q$ allows us to find a sharp control on  $|\co(A)\triangle \co(B)|$, from which the final result follows. We summarize the steps of the proof in the next subsection.
\subsection{Outline of the proof of \Cref{main_thm_5}}

The proof of \Cref{main_thm_5} follows the following steps.
\begin{enumerate}
    \setcounter{enumi}{-1}
    \item Reduce to the case that $A$ and $B$ are sandwiched between two balls of comparable sizes, and look like cones centered at the same vertex
    \item Let $C_A\supset A$ and $C_B\supset B$ be convex sets of size $|C_A|=|C_B|=(1+\gamma)|A|$, and let $T:C_A\to C_B$ be the optimal transport map between them.
    \item Note that if we let $E:=T^{-1}(B)\cap A$, then 
    $$(\delta+2\gamma) |A|\geq \left|tA+(1-t)B\right|-|E|\geq \int_E \Big(\det D\left(tId+(1-t)T\right)-1\Big)dx,$$
    where $D$ is the Jacobian.
    \item Analysing the eigenvalues of $D(T)$ (cf \Cref{lambdabound} akin to the methods in \cite{Figalli09}), we find that this implies
    $$\int_E ||D(T-Id)||_{op}dx\leq O_n\left(\sqrt{\frac{\delta+\gamma}{t}}\right)|A|.$$
    \item \label{elregstep}By an elliptic regularity argument (cf \Cref{lem_ellipticregularity}) this implies
    $||D(T-Id)(x)||_{op}\leq O_n\left(\sqrt{\frac{\delta+\gamma}{t}}\right)$ for points $x\in (1-\varepsilon)C_A$ and in particular in some small ball around the origin.
    \item Next, we note that $C_A\setminus E$ is small, so when integrating a bounded function, we find
    $$\int_{C_A}\frac{x^T}{||x||_2} (D(Id-T)(x))\frac{x}{||x||_2}dx\leq O_n(\gamma)|A|+\int_{E}||D(Id-T)(x)||_{op}dx\leq O_n\left(\sqrt{\frac{\delta+\gamma}{t}}\right)|A|.$$
    (Here, we crucially use that $DT$ is nonnegative definite; in particular, we only control the integral on the left-hand side from above.)
    \item Combining the two previous steps, we find 
    $$\int_{C_A}\frac{\frac{x^T}{||x||_2} (D(Id-T)(x))\frac{x}{||x||_2}}{||x||_2^{n-1}}dx\leq O_n\left(\sqrt{\frac{\delta+\gamma}{t}}\right)|A|.$$
    \item This allows us to integrate radially, giving 
    $$\int_{\partial C_A}\left\langle (x-T(x))-(o-T(o)),\frac{x}{||x||_2}\right\rangle dx\leq O_n\left(\sqrt{\frac{\delta+\gamma}{t}}\right)|A|.$$
    \item Up to this point, we only used that $B(o,\Omega(n))\in (1-\varepsilon)C_A$. So, in fact, for all $o'\in (1-2\varepsilon)C_A$ we get
    $$\int_{\partial C_A}\left\langle (x-T(x))-(o'-T(o')),\frac{x-o'}{||x-o'||_2}\right\rangle dx\leq O_n\left(\sqrt{\frac{\delta+\gamma}{t}}\right)|A|.$$
    \item Using the fact that $A$ and $B$ look like cones at the same vertex, we find that $|o'-T(o')|=O_n\left(\sqrt{\frac{\delta+\gamma}{t}}\right)$ (see \Cref{prop_conelike_0}).
    \item Hence, we find 
    $$\int_{\partial C_A}\left\langle x-T(x),\frac{x-o'}{||x-o'||_2}\right\rangle dx\leq O_n\left(\sqrt{\frac{\delta+\gamma}{t}}\right)|A|.$$
    This is the conclusion of \Cref{mainpropmass}.
    \item We find the same result (cf \Cref{maincormass}) if we first apply an affine transformation $Q$
    $$\int_{x\in\partial C_A} \left\langle Q(x)-T_Q(Q(x)),\frac{Q(x)-Q(o')}{||Q(x)-Q(o')||_2}\right\rangle dx\leq  O_n\left(\sqrt{\frac{\delta+\gamma}{t}}\right)|A|,$$

    \item \Cref{probabilisticlem} shows that, considering an appropriately distributed random affine transformation and random point $o'\in (1-\varepsilon)C_A$, then 
    $$\mathbb{E}_{Q,o'}\left[ \left\langle Q(x)-T_Q(Q(x)),\frac{Q(x)-Q(o')}{||Q(x)-Q(o')||_2}\right\rangle \right]\geq  \Omega_n(d(x, C_B)).$$
    \item \Cref{sdvsdistanceint} shows that 
    $$|C_A\triangle C_B|\leq O_n\left(\int_{\partial C_A}d(x,C_B)dx\right).$$
    \item Combining the last three steps gives the desired estimate
    $$|A\triangle B|\leq |C_A\triangle C_B|+ 2\gamma|A|\leq O_n\left(\sqrt{\frac{\delta+\gamma}{t}}\right)|A|.$$
\end{enumerate}

\section{Initial reduction}

We start with a simple reduction (\Cref{transportinitialreduction}) to allow us to assume that $A$ and $B$ are sandwiched between two balls of comparable sizes, and look like cones centered at the same vertex (cf \Cref{conelike}). Much of this section follows the lines of section 2 in \cite{BMStab}.

\subsection{Setup}

\begin{defn}\label{defn_cone}
A convex set $C\subset \mathbb{R}^n$ is called a \emph{cone} if there exists a hyperplane $H$ not containing the origin and a bounded convex set $P \subset H$ such that $$C=\bigsqcup_{t \geq 0} tP.$$
\end{defn}

\begin{defn}
We write $S^{v_0, \dots, v_n}$ for the simplex with vertices $v_0, \dots, v_n\in\mathbb{R}^n$.  Assuming that $S^{v_0, \dots, v_n}$ contains the origin in the interior, construct the family of cones $\mathfrak{C}^{v_0,\dots, v_n}:=\{C_i:0\leq i\leq n\}$, where
 $$C_i=\bigsqcup_{t\geq 0}t\co(v_0, \dots v_{i-1}, v_{i+1}, \dots, v_n).$$
\end{defn}
Note that the cones in $\mathfrak{C}^{v_0,\dots, v_n}$ form an essential partition of $\mathbb{R}^n$.

\begin{defn}\label{defn_simplex}
Fix vectors $e_0, \dots, e_n \in \mathbb{R}^n$ such that  $S^{e_0, \dots, e_n}$ is a regular unit volume simplex centered at the origin. Denote $S=S^{e_0, \dots, e_n}$ and $\mathfrak{C}=\mathfrak{C}^{e_0,\dots, e_n}$.
\end{defn}

\begin{defn}\label{defn_lambdabounded}
A pair of sets $X,Y\subset\mathbb{R}^n$ is \emph{$\lambda$-bounded} if there exists an $r>0$ so that 
$$rS\subset X,Y\subset  \lambda r S.$$
\end{defn}

\begin{defn}\label{defn_lambdaconebounded}
Given a cone $F\subset C'\in\mathfrak{C}$, a pair of sets $X,Y\subset F$ is \emph{$(\lambda,F)$-bounded}  if there exists an $r>0$ so that 
$$r(F\cap S)\subset X,Y\subset  \lambda r (F\cap S).$$
\end{defn}

\begin{defn}
A pair of sets $X,Y\subset \mathbb{R}^n$ is called a \textit{$\eta$-sandwich} if there exists a convex set $P$ such that $o\in P \subset X,Y \subset (1+\eta) P$.
\end{defn}
Note that given a cone $F$ and a $\lambda$-bounded  $\eta$-sandwich $X,Y\subset\mathbb{R}^n$, the pair $X\cap F,Y\cap F$ is also a $(\lambda,F)$-bounded  $\eta$-sandwich.

\begin{defn}\label{conelike}
Say sets $A,B\subset \mathbb{R}^n$ are $(\gamma,\ell,\lambda,\mu)$ \emph{conelike} if there exist convex sets $C_A\supset A,C_B\supset B$ with $|C_A|=|C_B|=(1+\gamma)|A|=(1+\gamma)|B|$, a convex set $K$, and a set $S''$ obtained by intersecting a cone with a half space with the following properties
\begin{enumerate}
    \item $B(o,1/\ell)\subset C_A,C_B\subset B(o,\ell)$,
    \item $S'' \subset A-z, C_A-z, B-z, C_B-z \subset \lambda S''$, for some $z\in \mathbb{R}^n$, and
    \item $K\subset A+x, C_A+x, B+y, C_B+y \subset (1+\mu)K$, for some $x,y\in\mathbb{R}^n$.
\end{enumerate}
\end{defn}

\subsection{Proposition}
\begin{prop}\label{transportinitialreduction}
Assume that for sets $A,B\subset\mathbb{R}^n$ satisfying the conditions of \Cref{main_thm_5} that are $(\gamma,\ell,\lambda,\mu)$ conelike (for $\mu$ sufficiently small in terms of $n,t,\ell,$ and $\lambda$), we have $|A \triangle B| \leq c_{n,\ell,\lambda}\sqrt{\frac{\delta+\gamma}{t}}|A|.$ 
Then \Cref{main_thm_5} is true for all set $A,B\subset\mathbb{R}^n$.
\end{prop}

\subsection{Lemmas}

We recall the following result by Michael Christ. 
\begin{thm}[Christ 2012, \cite{christ2012near}]\label{thm_christ}
For all $n \in \mathbb{N}$, $t \in (0,1)$ and $\eta>0$, there exist constants $d^{\ref{thm_christ}}>0$, so that for all measurable $X,Y\subset\mathbb{R}^n$ of equal volume with the property that $|tX+(1-t)Y|< (1+d^{\ref{thm_christ}})|X|$, then $$\min_{v\in\mathbb{R}^n}|\co(X\cup (v+Y))|\leq (1+\eta)|X|.$$
\end{thm}

We also use the following three lemmas from \cite{BMStab}

\begin{lem}[Proposition 5.4 in \cite{BMStab}]
\label{lem_first_step}
Let $v_0,\dots, v_n\subset\mathbb{R}^n$ be vectors not contained in a halfspace and let $A,B\subset \mathbb{R}^n$ be measurable sets with equal volume. 
Then there exists a vector $v \in \mathbb{R}^n$ such that for every 
cone $C\in\mathfrak{C}^{v_0,\dots, v_n}$ we have $$|A \cap C|=|(B+v)\cap C|.$$
Moreover, for every $\eta,\lambda>0$, there is a computable constant $\eta'^{\ref{lem_first_step}}>0$ such that the following holds. If $\{v_0, \dots, v_n\}= \{e_0, \dots, e_n\}$ (as in \Cref{defn_simplex}) and if $A,B\subset \mathbb{R}^n$ is a $\lambda$-bounded $\eta'^{\ref{lem_first_step}}$-sandwich, then $A,B+v$ is a $2\lambda$-bounded $\eta$-sandwich.
\end{lem}

We won't use \Cref{thm_christ} directly, but only through \Cref{prop_sandwichmaker}.
\begin{lem}[Lemma 2.11 in \cite{BMStab}]\label{prop_sandwichmaker}
For $n \in \mathbb{N}$, $t \in (0,1/2]$ and $\eta>0$, there exist constants $c^{\ref{prop_sandwichmaker}}$  and $d^{\ref{prop_sandwichmaker}}_{n,t}(\eta)>0$ so that the following holds. If $X,Y\subset\mathbb{R}^n$ are measurable sets with $|X|=|Y|$ and $|tX+(1-t)Y|= (1+\delta)|X|$ with $\delta \in[0,d^{\ref{prop_sandwichmaker}}_{n,t}(\eta))$, then, up to translation, there exist measurable sets $ X',Y' \subset \mathbb{R}^n$  so that 
\begin{enumerate}
\item $X',Y'$ is an $\eta$-sandwich,
\item $|X'|=|Y'|=|X|$,
\item $\co(X')=\co(X)$ and $\co(Y')=\co(Y)$, 
\item $|X'\triangle X|+|Y'\triangle Y|\leq c^{\ref{prop_sandwichmaker}}t^{-1}\delta|X|$,
\item $|tX'+(1-t)Y'|\leq (1+\delta)|X|$.
\end{enumerate}
Moreover, if $X\subset Y$, we additionally find $X'\subset Y'$.
\end{lem}

\begin{lem}[Lemma 2.12 in \cite{BMStab}]\label{prop_boundedmaker}
For $n \in \mathbb{N}$, and $\eta>0$ the following holds. If $X,Y\ \subset \mathbb{R}^n$ is an $\eta$-sandwich, then there exists $v\in \mathbb{R}^n$ and there exists a linear transformation $\theta\colon \mathbb{R}^n\rightarrow \mathbb{R}^n$ such that $\theta(v+X), \theta(v+Y)$ is a $(n^2+n^3\eta)$-bounded $n\eta$-sandwich.
\end{lem}

\subsection{Proof of \Cref{transportinitialreduction}}

\begin{proof}[Proof of \Cref{transportinitialreduction}] First note that we may assume $|A|=|B|=1$.
Let $\lambda=\lambda_n:=16n^6$. Let $\ell'_n$ be minimal, so that a translate of $B(0,1/\ell'_n)$ is contained in $\frac{1}{4n^3}S\cap C_0$, where $S$ and $C_0$ are defined in \Cref{defn_simplex}. Let $\ell''_n$ be minimal, so that $4n^3S\cap C_0$ is contained in some translate of $B(o,\ell''_n)$. Let $\ell_n:=2\max\{\ell'_n,\ell''_n\}$. Find $\mu=\mu_{n,t}:=\mu_{n,t,\ell_n,\lambda_n}$ sufficiently small as required by the assumption. Let $\eta=\eta_{n,t}:=\mu$ for notational consistency.  Choose $\eta'$ to be sufficiently small in terms of $\eta$ and $n$, so that the second part of \Cref{lem_first_step} applies with parameters $(\eta'_n)^{\ref{lem_first_step}}=\eta'$, $\eta^{\ref{lem_first_step}}_n=\eta$ and $\lambda_n^{\ref{lem_first_step}}=2n^3$. Choose $d_{n,t}$ smaller than the constant $d_{n,t}(\eta'_{n,t})$ in \Cref{prop_sandwichmaker}. Let $c_n:=(n+1)c_{n,\ell_n,\lambda_n}\sqrt{w_n}+c^{\ref{prop_sandwichmaker}},$ where $w_n=(n+1)(4n^3)^n$, $c_{n,\ell_n,\lambda_n}$ is the constant from the assumption and $c^{\ref{prop_sandwichmaker}}$ is the constant from \Cref{prop_sandwichmaker}.

First, use \Cref{prop_sandwichmaker} with parameter $\eta'$ to find $A^1,B^1$ which form an $\eta'$-sandwich. Note that 
$$|A\triangle B|\leq \left|A\triangle A^1\right|+\left|B\triangle B^1\right|+\left|A^1\triangle B^1\right|\leq \left|A^1\triangle B^1\right|+c^{\ref{prop_sandwichmaker}}t^{-1}\delta |A|,$$
so it suffices to show $$\left|A^1\triangle B^1\right|\leq c_n \sqrt{\frac{\delta+\gamma}{t}}|A^1|.$$
Now apply \Cref{prop_boundedmaker} to $A^1,B^1$, to find $A^2,B^2$ an $n^2+n^3\eta'$-bounded $\eta'$ sandwich. $A^2,B^2$ are just a linear transformation and a translation away from $A^1,B^1$. Note that $n^2+n^3\eta'\leq 2n^3$, so, in particular, $A^2,B^2$ is a $2n^3$-bounded $\eta'$ sandwich.

We then apply \Cref{lem_first_step} with vectors $e_0,\dots,e_n$ and cones $\mathfrak{C}$ from \Cref{defn_simplex}. Let $A^3=A^2$ and $B^3$ be the translation of $B^2$ given by the lemma. Note that by definition of $\eta'$, we find that $A^3,B^3$ is a $4n^3$-bounded $\eta$-sandwich with the property that $|A^3\cap C|=|B^3\cap C|$ for all $C\in\mathfrak{C}$.

Fix a $C\in \mathfrak{C}$, and let 
$$A':= A^3\cap C\text{ and }B':= B^3\cap C.$$
We will show that $A'$ and $B'$ are of the correct form to bound their symmetric difference.

Note that $t(A^3\cap C')+(1-t)(B^3\cap C')\subset(tA^3+(1-t)B^3)\cap C'$ so that these are disjoint for different $C'\in\mathfrak{C}$. Hence, by Brunn-Minkowski we find that
$$ (1+\delta)|A|\geq|tA^3+(1-t)B^3|\geq \sum_{C'\in\mathfrak{C}}|t(A^3\cap C')+(1-t)(B^3\cap C')|\geq |tA'+(1-t)B'|+ \sum_{C\neq C'\in\mathfrak{C}}|A\cap C'|.$$
In particular, we find
$|tA'+(1-t)B'|\leq |A'|+\delta |A|.$
Since $A^3,B^3$ is $4n^3$-bounded, there exists some  $r>0$ so that 
$rS\subset A^3,B^3\subset 4n^3rS.$
Given that $|A^3|=|A|=1$ and $|S|=1$, this implies $1/4n^3\leq r\leq 1$, and thus $\frac{1}{4n^3}S\subset A^3,B^3\subset 4n^3S.$
Since $A^3,B^3$ is a $\eta$-sandwich, there exists a convex set
$K\subset A^3,B^3\subset (1+\eta)K.$

Thus, we find that $|A'|\geq \left|\frac{1}{4n^3}S\cap C\right|= \frac{1}{(n+1)(4n^3)^n}|A|$ for all $C\in\mathfrak{C}$. For notational convenience, let $w_n=(n+1)(4n^3)^n$, so that $$|tA'+(1-t)B'|\leq (1+w_n\delta)|A'|.$$
With this bound on $|A'|$ and $|B'|$ in hand, we are ready to define $C_{A'}$ and $C_{B'}$. Note that $\co(A')\setminus A'\subset \co(A^3)\setminus A^3$ and $|\co(A^3)\setminus A^3|=|\co(A)\setminus A|$, so that
$|\co(A')|\leq |A'|+\left|\co\left(A^3\right)\setminus A^3\right|=|A'|+\gamma |A^3|\leq (1+w_n\gamma)|A'|,$
and analogously $|\co(B')|\leq (1+w_n\gamma)|B'|$. Find convex sets $C_{A'}\supset \co(A')$ and $C_{B'}\supset \co(B')$ so that $$C_{A'},C_{B'}\subset C\cap 4n^3S\cap (1+\eta)K\text{ and }|C_{A'}|=|C_{B'}|\leq(1+w_n\gamma)|A'|=(1+w_n\gamma)|B'|.$$

With these in place we check that these sets are conelike (cf \Cref{conelike}).
Recall that
$\frac{1}{4n^3}S\cap C\subset A',B',C_{A'},C_{B'}\subset 4n^3S\cap C.$
It's easy to see that $\frac{1}{4n^3}S\cap C$ contains a translate of $B(o,1/\ell'_n)$ and $4n^3S\cap C$ is contained in a translate of $B(o,\ell''_n)$. Hence, we find that $A',B',C_{A'},C_{B'}$ satisfy the first condition in \Cref{conelike}.

For the second condition note that $S'':=\frac{1}{4n^3}S\cap C$ is indeed a set obtained by intersecting a cone with a halfspace. Hence, if we recall that $\lambda_n=(4n^3)^2$, and let $y=o$, we find that $A',B',C_{A'},C_{B'}$ satisfy the second condition in \Cref{conelike}.

For the third condition, note that
$K\cap C\subset A'\subset C_{A'}\subset (1+\eta)(K\cap C),$
and analogously for $B',C_{B'}$. Hence, if we let $\eta_n$ sufficiently small in terms of $n,t,\ell_n$ and $\lambda_n$, set $x=y=o$, we find that $A',B',C_{A'},C_{B'}$ satisfy the third condition in \Cref{conelike}.

Hence, $A',B'$ are $(w_n\gamma,\ell_n,\lambda_n,\mu_n)$ conelike, with $\mu_n$ sufficiently small in terms of $n,t,\ell_n,$ and $\lambda_n$ so that by the assumption in the proposition, we have
$$\left|\left(A^3\cap C\right)\triangle \left(B^3\cap C\right)\right|=|A'\triangle B'|\leq c_{n,\ell_n,\lambda_n}\sqrt{\frac{w_n\delta+w_n\gamma}{t}}|A'|\leq \frac{c_{n}-c^{\ref{prop_sandwichmaker}}}{n+1}\sqrt{\frac{\delta+\gamma}{t}}|A|,$$
where we recall that  $c_n:=(n+1)c_{n,\ell_n,\lambda_n}\sqrt{w_n}+c^{\ref{prop_sandwichmaker}}$.
We conclude by adding up the contributions from the different cones $C\in\mathfrak{C}$.
$$\left|A^1\triangle B^1\right|=\left|A^3\triangle B^3\right|=\sum_{C\in\mathfrak{C}}\left|\left(A^3\triangle B^3\right)\cap C\right|\leq  (c_n-c^{\ref{prop_sandwichmaker}})\sqrt{\frac{\delta+\gamma}{t}}|A|.$$
We conclude with the previous note that $$|A\triangle B|\leq |A^1\triangle B^1|+c^{\ref{prop_sandwichmaker}}t^{-1}\delta |A|\leq c_n\sqrt{\frac{\delta+\gamma}{t}}|A|.$$
\end{proof}

\section{Intermediate propositions}

The proof of  \Cref{main_thm_5} relies on optimal transport.
For the purpose of this paper, we only need to know the following classical result: given two bounded sets $A,B\subset \mathbb{R}^n$ of positive volume, there exists a convex function $\varphi:\mathbb{R}^n \to \mathbb{R}$ whose gradient sends the normalized indicator function of $A$ onto that of $B$. More precisely, if we define $T:=\nabla\varphi$, then $T_\sharp\left(\frac{1}{|A|}\mathbf 1_{A}\right)=\frac{1}{|B|}\mathbf 1_{B}$, where $T_\sharp$ denotes the push-forward through the map $T$. Furthermore, this map is unique: If $\varphi_1$ and $\varphi_2$ are two convex functions such that $T_i:=\nabla\varphi_i$ sends $\frac{1}{|A|}\mathbf 1_{A}$ to $\frac{1}{|B|}\mathbf 1_{B}$, then $T_1=T_2$ a.e. inside $A$.

From now on, whenever we say that $T$ is the optimal transport from $A$ to $B$, we mean the (unique) gradient of a convex function that sends $\frac{1}{|A|}\mathbf 1_{A}$ to $\frac{1}{|B|}\mathbf 1_{B}$. We refer to \cite[Chapter 4.6]{FigMAbook} for a quick introduction to this beautiful theory and more references.

\subsection{Propositions}

\begin{prop}\label{mainpropmass}
For every $n\in \mathbb{N}$ and all $t,\varepsilon, \lambda, \ell, \delta, \gamma>0$ with $\delta+\gamma\leq t^{2n-1}/2$, there exists $c^{\ref{mainpropmass}}_{n, \varepsilon, \lambda, \ell}, \mu^{\ref{mainpropmass}}_{n, t, \varepsilon, \ell}>0$ such that the following holds. Assume that $A,B \subset \mathbb{R}^n$ are $(\gamma,\ell,\lambda,\mu^{\ref{mainpropmass}}_{n, t, \varepsilon, \ell})$ conelike. Moreover, assume that $ \left|tA+(1-t)B\right| \leq (1+\delta)|A|$.  If $T:C_A\to C_B$ is the optimal transport map from $C_A$ to $C_B,$ then
\begin{align*}\int_{x\in\partial C_A}&\max\left\{\left\langle x-T(x),\frac{x-o'}{||x-o'||_2}\right\rangle, 0\right\} dx\leq  c^{\ref{mainpropmass}}_{n,\varepsilon,\lambda,\ell}\sqrt{\frac{\delta+\gamma}{t}}|A|,
\end{align*}
for any $o'\in(1-\varepsilon)C_A$.
\end{prop}

\begin{cor}\label{maincormass}
In addition, for all $\theta>0$ there exists $c_{n,\varepsilon, \lambda, \ell, \theta}$ such that the following holds.  Let  $Q:\mathbb{R}^n\to\mathbb{R}^n$ be an an affine transformation with $||Q||_{op}, ||Q^{-1}||_{op}\leq \theta$. If $T_Q:Q(C_A)\to Q(C_B)$ is the optimal transport map from $Q(C_A)$ to $Q(C_B)$, then
\begin{align*}\int_{x\in\partial C_A}\max\left\{ \left\langle Q(x)-T_Q(Q(x)),\frac{Q(x)-Q(o')}{||Q(x)-Q(o')||_2}\right\rangle, 0 \right\} dx\leq  c_{n,\varepsilon, \lambda, \ell, \theta}\sqrt{\frac{\delta+\gamma}{t}}|Q(A)|,
\end{align*}
for all $o'\in (1-\varepsilon)C_A$.\qed
\end{cor}

\begin{defn}
Given a parameter $\theta$, let a \emph{random scaling} be the random affine transformation $Q\sim \mathcal{Q}_\theta$ generated as follows. Sample a uniformly random orthonormal basis $e_1,\dots,e_n\in\mathbb{R}^n$ and sample $\theta_1,\dots,\theta_n \in [\theta^{-1},\theta]$ i.i.d. uniformly. Then, in this basis, let $Q$ be the random transformation given by the diagonal matrix with entries $\theta_i$.
\end{defn}

\begin{prop}\label{probabilisticlem}
For every $n\in \mathbb{N}$, $\ell>1$, there exists constants $\theta=\theta_{n,\ell}, c_{n,\ell}>0$ such that if 
\begin{itemize}
    \item $B(o,1/\ell)\subset \frac12C_A,\frac12C_B\subset B(o,\ell)$ where $C_A$ and $C_B$ are convex,
    \item for every affine tranformation $Q\colon \mathbb{R}^n\to\mathbb{R}^n$, $T_Q$ is a map with $T_Q(Q(C_A))\subset Q(C_B)$,
    \item $Q\sim \mathcal{Q}_{\theta}$ is a random scaling, and 
    \item $o'$ is chosen uniformly random from $B(o,1/\ell)$,
\end{itemize}
then, for all $x\in \partial C_A$,
\begin{align*}
    \mathbb{E}_{Q,o'}\left[ \max\left\{ \left\langle Q(x)-T_Q(Q(x)),\frac{Q(x)-Q(o')}{||Q(x)-Q(o')||_2}\right\rangle, 0 \right\} \right]&\geq  c_{n,\ell} d(x, C_B).
\end{align*}
\end{prop}

\begin{prop}\label{sdvsdistanceint}
For all $n\in\mathbb{N}, \ell\geq1$, there exists constants $c_{n,\ell}$, so that given two convex sets $X,Y\subset\mathbb{R}^n$ of equal volume with $B(o,1/\ell)\subset X,Y\subset B(o,\ell)$ we have
$$|X\triangle Y|\leq c_{n,\ell}\int_{\partial X}d(x,Y)dx.$$
\end{prop}

\subsection{Auxiliary Lemmas}

\begin{lem}\label{prop_conelike_0}
For every $n\in \mathbb{N}$ and $\lambda, \ell>0$, there exists $\varepsilon^{\ref{prop_conelike_0}}_{n, \lambda, \ell}, m^{\ref{prop_conelike_0}}_{n, \lambda, \ell}, r^{\ref{prop_conelike_0}}_{n, \lambda, \ell}, \sigma^{\ref{prop_conelike_0}}_{n, \lambda, \ell}>0$ such that the following holds.  Say sets $A,B\subset \mathbb{R}^n$ are $(\gamma,\ell,\lambda,\mu)$ conelike. Then for every non-zero $y_2 \in \mathbb{R}^n$, there exists $s \in \{\pm 1\}$ such that for every map $M \colon \partial C_A \cap \partial C_B \rightarrow \partial C_A \cup \partial C_B$ the following holds. There exist faces $F_A$ of $C_A$ and $F_B$ of $C_B$ with the same supporting hyperplane $H$, and there exists $w_0 \in H$, such that  $$B^n(w_0,1/r^{\ref{prop_conelike_0}}_{n, \lambda, \ell}) \cap H \subset F_A \cap F_B.$$
Moreover, for every $w \in B^n(w_0,1/r^{\ref{prop_conelike_0}}_{n, \lambda, \ell}) \cap H $ there exists a ball $X_w \subset \mathbb{R}^n$ such that with $y_1= M(w)-w$ we have
\begin{enumerate}
    \item  $X_w \subset (1-\varepsilon^{\ref{prop_conelike_0}}_{n, \lambda, \ell})(C_A\cap C_B)$
    \item $|X_w| \geq  m^{\ref{prop_conelike_0}}_{n, \lambda, \ell}$
    \item $d(w, X_w) \geq 1/(4r^{\ref{prop_conelike_0}}_{n \lambda, \ell}) $
    \item $\mathbb{P}_{x \in X_w} \left(\frac{\langle y_1, x-w \rangle}{|y_1||x-w|} \geq 0\right)\geq 1/2$
    \item $\mathbb{P}_{x \in X_w} \left(\frac{\langle sy_2, x-w \rangle}{|y_2||x-w|} \geq \sigma^{\ref{prop_conelike_0}}_{n, \lambda, \ell}\right) =1$.
\end{enumerate}
\end{lem}

\begin{lem}\label{lem_conelike_1}
For every $n\in \mathbb{N}$ and $\lambda, \ell>0$, there exists $r^{\ref{lem_conelike_1}}_{n, \lambda, \ell}>\ell$ such that the following holds.  Say sets $A,B\subset \mathbb{R}^n$ are $(\gamma,\ell,\lambda,\mu)$ conelike. Construct simplex $S'=S+z$ with a vertex at $z$. Let $F_0, F_1, \dots F_n$ be the faces of $S'$ where $F_0$ is the face opposite vertex $z$. Then
\begin{enumerate}
    \item $ S' \subset C_A \cap C_B$
    \item $F_1 \cup \dots \cup F_n \subset \partial C_A \cap \partial C_B$.
    \item $B(u, 1/r^{\ref{lem_conelike_1}}) \subset S' \subset B(u, r^{\ref{lem_conelike_1}}) $ for some $u \in \mathbb{R}^n$.
\end{enumerate}
\end{lem}

\begin{lem}\label{lem_conelike_final}
For every $n \in \mathbb{N}$ and $r, \sigma>0$, there exists $ k^{\ref{lem_conelike_final}}_{n,r,\sigma}>0$ such that the following holds. Let $H$ be a hyperplane and let $H^+$ and $H^-$ be the two half-spaces determined by $H$. Let $w \in H$ with $|w| \leq r$. Let $f$ be the normal vector of $H$ pointing to $H^+$. Let $y_1, y_2$ be two vectors such that $\frac{\langle y_1, f \rangle}{|y_1||f|} \geq 0$ and $\frac{\langle y_2, f \rangle}{|y_2||f|} \geq \sigma$. Then the ball
$$X= B^n(w+f/(2r) , k),$$ has the following properties:
\begin{enumerate}
    \item $X \subset (1-1/(8r^2))(B^n(w,1/r) \cap H^+)$
    \item $\mathbb{P}_{x \in X} \left(\frac{\langle y_1, x-w \rangle}{|y_1||x-w|} \geq 0\right)\geq 1/2$ 
    \item For all $x \in X$ we have
$\frac{\langle y_2, x-w \rangle}{|y_2||x-w|} \geq \sigma/4$
    \item For all $x \in X$ we have $|x-w| \geq 1/4r.$
\end{enumerate}
\end{lem}

\begin{lem}\label{lem_conelike_1.5}
For every $n \in \mathbb{N}$ and $r>1$, the following holds. Let $S'$ be a simplex such that $B^n(u, 1/r) \subset S' \subset B^n(u, r)$ for some $u \in \mathbb{R}^n$. Let $F$ be a facet of $S'$ and let $H$ be the supporting hyperplane of $F$. Finally, let $H^+$ and $H^-$ be the two half spaces determined by $H$, such that $H^+$ contains $S'$ and $H^-$ is disjoint from the interior of $S'$. Then there exists $v \in F$ such that 
$$B^n(v,1/r)\cap H^+ \subset S'$$
and
$$B^n(v, 1/r) \cap H \subset F \subset B^n(v,2r) \cap H$$
\end{lem}

\begin{lem}\label{lem_conelike_3}
For every $n \in \mathbb{N}$ and $r>0$, there exists $\sigma^{\ref{lem_conelike_3}}=\sigma^{\ref{lem_conelike_3}}_{n,r}>0$ such that the following holds. Let $S'$ be a simplex such that $B^n(u, 1/r) \subset S' \subset B^n(u, r)$ for some $u \in \mathbb{R}^n$. Let $f_0,f_1, \dots, f_n$ be the inward normal vectors to the faces of $S'$. Then for every unit vector $v \in \mathbb{R}^n$ there exists $1 \leq i \leq n$ such that
$$ |\langle f_i, v  \rangle| \geq \sigma^{\ref{lem_conelike_3}} .$$
\end{lem}

\begin{lem}\label{lambdabound}
For every $n \in \mathbb{N}$, there exists a constant $c^{\ref{lambdabound}}_n>0$ such that the following holds. If $0<t\leq 1/2$ and $\lambda_1, \dots, \lambda_n>0$ and $\lambda_1\cdots\lambda_n=1$, then
\begin{align*}
\sqrt{\sum_i (\lambda_i-1)^2}\leq c^{\ref{lambdabound}}_nt^{-n}\bigg(\prod_i (t+(1-t)\lambda_i)-1\bigg)+ c^{\ref{lambdabound}}_n t^{-\frac{1}{2}}\sqrt{\prod_i (t+(1-t)\lambda_i)-1}.
\end{align*}
\end{lem}

\begin{lem}\label{lem_filling}
For every $n \in \mathbb{N}$ and for all $t,\varepsilon, \ell>0$, there exists $\mu=\mu_{n,t, \varepsilon, \ell }>0$ such that the following holds.  Assume that $A,B \subset \mathbb{R}^n$ are $(\gamma,\ell,\lambda,\mu)$ conelike. Then $$tA+(1-t)B \supset t(1-\varepsilon/4)C_A+(1-t)C_B.$$ 
\end{lem}

\begin{lem}\label{lem_ellipticregularity}
Let $C_A,C_B$ be two convex sets in $\mathbb{R}^n$ with equal volume $1$ and satisfying
\begin{equation}
\label{eq:normalized}
B^n(o,1/R)\subset C_A,C_B\subset B^n(o,R) \qquad \text{for some constant $R>1$}.
\end{equation}
Let $T=\nabla \varphi$ denote the optimal transport map from ${C_A}$
to ${C_B}$.
Then, for every $\varepsilon \in (0,1)$, 
$$
\|D(T-{\rm Id})\|_{L^\infty((1-\varepsilon)C_A)} \leq C_{n,R,\varepsilon} \|D(T-{\rm Id})\|_{L^1((1-\varepsilon/2)C_A)}.
$$
\end{lem}

\subsection{Proofs of propositions}

\subsubsection{Proof of \Cref{mainpropmass}}

\begin{proof}[Proof of \Cref{mainpropmass}]
We first observe that, by Caffarelli's regularity theory \cite{Caff92int,Caff92bdry}, we can write $T=\nabla \varphi$, where the function $\varphi \colon C_A \rightarrow \mathbb{R}^n$ is a smooth strictly convex solution of $\det D^2\varphi=1$. 

Therefore, for $x \in C_A$, $DT(x)=D^2\varphi(x)$ is a positive definite symmetric matrix with determinant $1$ and its eigenvalues $\lambda_1(x), \lambda_2(x), \dots, \lambda_n(x)$ satisfy
\begin{align}\label{mass_1}
    \lambda_1(x), \lambda_2(x), \dots, \lambda_n(x)>0 \text{ and } \lambda_1(x) \lambda_2(x) \cdots \lambda_n(x)=1.
\end{align}

Note that we can write $tId+(1-t)T=\nabla( \frac{t}{2}||x||_2^2+(1-t)\varphi)$ and that the function $\frac{t}{2}||x||_2^2+(1-t)\varphi$ is also strictly convex.

Therefore, for $x \in C_A$, $$D(tId+(1-t)T)(x)=D^2\left(\frac{t}{2}||x||_2^2+(1-t)\varphi\right)(x)$$ is a positive definite symmetric matrix with eigenvalues $$t+(1-t)\lambda_1(x), t+(1-t) \lambda_2(x), \dots, t+(1-t) \lambda_n(x).$$ In particular, the function $tId+(1-t)T\colon C_A \to \mathbb{R}^n$ is injective.

The above discussion shows that for every compact subset $E\subset C_A$ we have
\begin{align*}
    \bigg|\bigcup_{x \in E} tx+(1-t)T(x)\bigg|= \int_E \det D\left(tId+(1-t)T\right)dx=\int_E  \biggl(\prod_i t+(1-t)\lambda_i\biggr) dx
\end{align*}
Construct the set $E:=(T^{-1}(B)\cap A)\cup (1-\varepsilon/4)C_A$. By  \Cref{lem_filling}, it follows that 
$tA+(1-t)B \supset \bigcup_{x \in E} tx+(1-t)T(x)$. By hypothesis, we have 
$|tA+(1-t)B| \leq (1+\delta)|A|.$ 

Combining the last three lines, we get
\begin{align*}
      \int_E  \biggl(\prod_i t+(1-t)\lambda_i\bigg) dx \leq (1+\delta)|A|
\end{align*}
By hypothesis, we also have
\begin{align}\label{mass1.3}
|C_A \setminus A| = |C_B \setminus B| = \gamma |A|,  
\end{align}
and because $T$ is bijective and measure preserving, we get 
\begin{align}\label{mass1.5}
      |E| \geq |T^{-1}(B) \cap A| = |A \setminus T^{-1}(C_B\setminus B)| \geq |A| - |C_B\setminus B| \geq (1-\gamma) |A|.
\end{align}
Combining the last two inequalities, we get
\begin{align}\label{mass_2}
     \int_E  \biggl[\prod_i (t+(1-t)\lambda_i)-1\biggr] dx \leq (\delta+\gamma)|A|
\end{align}
Also, Lemma \ref{lambdabound} together with \eqref{mass_1}  imply that
\begin{align}\label{mass_3}
\sqrt{\sum_i (\lambda_i-1)^2}\leq c^{\ref{lambdabound}}_nt^{-n}\bigg(\prod_i (t+(1-t)\lambda_i)-1\bigg)+ c^{\ref{lambdabound}}_nt^{-\frac{1}{2}}\sqrt{\prod_i (t+(1-t)\lambda_i)-1}.
\end{align}
Therefore, we get
\begin{align}\label{mass_4}
\begin{split}
    \int_E ||D(T-Id)||_{op}dx&\leq\int_E\sqrt{\sum_i (\lambda_i-1)^2}dx\\
&\leq \int_E \left(c_n^{\ref{lambdabound}}t^{-n}\bigg(\prod_i (t+(1-t)\lambda_i)-1\bigg)+ c_n^{\ref{lambdabound}}t^{-\frac{1}{2}}\sqrt{\prod_i (t+(1-t)\lambda_i)-1}\,\right)dx\\
&\leq c^{\ref{lambdabound}}_nt^{-n}(\delta+\gamma) |A|+c^{\ref{lambdabound}}_nt^{-\frac{1}{2}}\int_E \sqrt{\prod_i (t+(1-t)\lambda_i)-1}\,dx\\
&\leq c^{\ref{lambdabound}}_nt^{-n}(\delta+\gamma) |A|+c^{\ref{lambdabound}}_nt^{-\frac{1}{2}}\sqrt{|E|}\sqrt{\int_E \bigg(\prod_i (t+(1-t)\lambda_i)-1\bigg)dx}\\
&\leq c^{\ref{lambdabound}}_nt^{-n}(\delta+\gamma) |A|+c^{\ref{lambdabound}}_nt^{-\frac{1}{2}}\sqrt{\delta+\gamma}\sqrt{|E|}\sqrt{|A|}\\
&\leq c^{\ref{lambdabound}}_nt^{-n}(\delta+\gamma) |A|+2c^{\ref{lambdabound}}_nt^{-\frac{1}{2}}\sqrt{\delta+\gamma}|A|\leq  3c^{\ref{lambdabound}}_n\sqrt{\frac{\delta+\gamma}{t}}|A|.
\end{split}
\end{align}
Here, the first inequality follows from the fact that the operator norm is upper bounded by the Hilbert-Schmidt norm. The second inequality follows from \eqref{mass_3}. The third inequality follows from \eqref{mass_2}. The fourth inequality follows from the concavity of the function $x^{\frac{1}{2}}$. The fifth inequality follows again from \eqref{mass_2}. The sixth inequality follows from the hypothesis $|E|\leq |C_A|\leq 2|A|$ and the final inequality follows from the hypothesis $\delta+\gamma\leq t^{2n-1}/2$.

Thus, \Cref{lem_ellipticregularity}, together with \eqref{mass_4} and the fact that $E\supset (1-\varepsilon/4)C_A$,  implies that for $x\in (1-\varepsilon/2)C_A$
\begin{align}\label{mass_5}
||D(T-Id)(x)||_{op}\leq c^{\ref{lem_ellipticregularity}}_{n, \varepsilon, \ell}\sqrt{\frac{\delta+\gamma}{t}}
\end{align}
Now, fix $o' \in (1-\varepsilon)C_A$ and set $P=B(o',\varepsilon/2\ell)$. Note that as $C_A \supset B(o, \ell^{-1})$, it follows that $P \subset (1-\varepsilon/2)C_A \subset E$. Combining \eqref{mass_4} and \eqref{mass_5}, we deduce that
\begin{align}\label{mass_6}
    \begin{split}
    \int_E \frac{||D(T-Id)||_{op}}{||x-o'||_2^{n-1}}dx&\leq\int_{P} \frac{||D(T-Id)||_{op}}{||x-o'||_2^{n-1}}dx  +\int_{E\setminus P} \frac{||D(T-Id)||_{op}}{||x-o'||_2^{n-1}}dx\\
    &\leq c^{\ref{lem_ellipticregularity}}_{n, \varepsilon, \ell}\sqrt{\frac{\delta+\gamma}{t}}\int_{P} \frac{1}{||x-o'||_2^{n-1}}dx+\int_{E\setminus P} \frac{||D(T-Id)||_{op}}{||x-o'||_2^{n-1}}dx\\
    &\leq c^{\ref{lem_ellipticregularity}}_{n, \varepsilon, \ell}\sqrt{\frac{\delta+\gamma}{t}}\int_{B(o,\varepsilon/2\ell)} \frac{1}{||x||_2^{n-1}}dx  + (2\ell\varepsilon^{-1})^{n-1}\int_{E\setminus P}||D(T-Id)||_{op}dx\\
    &\leq c^{\ref{lem_ellipticregularity}}_{n, \varepsilon, \ell}\sqrt{\frac{\delta+\gamma}{t}}\frac{\varepsilon}{2\ell}|S^{n-1}(o,1)|  + (2\ell\varepsilon^{-1})^{n-1}\int_{E}||D(T-Id)||_{op}dx\\
    &\leq c^{\ref{lem_ellipticregularity}}_{n, \varepsilon, \ell}\sqrt{\frac{\delta+\gamma}{t}}\frac{\varepsilon}{2\ell}|S^{n-1}(o,1)|  + (2\ell\varepsilon^{-1})^{n-1}3c^{\ref{lambdabound}}_n\sqrt{\frac{\delta+\gamma}{t}}|A|\leq c^{\ref{mass_6}}_{n,\varepsilon, \ell}\sqrt{\frac{\delta+\gamma}{t}}|A|.
    \end{split}
\end{align}
Here, the first inequality is immediate. The second inequality follows from \eqref{mass_5} and the fact that $P \subset (1-\varepsilon/2)C_A$. The third inequality follows from the trivial bound that for $x \not \in P$ we have $||x-o'||_2^{-1}\leq 2\ell\varepsilon^{-1}$. The fourth inequality follows from a simple change of variables. The fifth inequality follows from \eqref{mass_4}. The final inequality follows from the hypothesis $|A|\geq 2^{-1}|B(o, \ell^{-1})|$.

In particular, by definition of the operator norm, 
\begin{align*}
    \int_E \max\left\{\frac{\frac{(x-o')^T}{||x-o'||_2} (D(Id-T)(x)) \frac{x-o'}{||x-o'||_2}}{||x-o'||_2^{n-1}},0\right\}dx \leq \int_E \frac{||D(T-Id)||_{op}}{||x-o'||_2^{n-1}}dx&\leq c^{\ref{mass_6}}_{n,\varepsilon, \ell}\sqrt{\frac{\delta+\gamma}{t}} |A|.
\end{align*}
Note now that, for $x \in C_A$, as the eigenvalues of $D(T)(x)$ are all positive by \eqref{mass_1}, it follows that the eigenvalues of $D(Id-T)(x)$ are at most $1$, which implies that
$$ \frac{(x-o')^T}{||x-o'||_2} (D(Id-T)(x)) \frac{x-o'}{||x-o'||_2} \leq 1. $$
As $P \subset (1-\varepsilon/2)C_A \subset E$, it follows that for $x \in C_A\setminus E$ we have $||x-o'||_2\geq \varepsilon/2\ell$, which implies that
$\frac{1}{||x-o'||_2^{n-1}} \leq (2\ell\varepsilon^{-1})^{n-1}.$

Combining the last three inequalities with \eqref{mass1.3} and \eqref{mass1.5}, we deduce
\begin{align}\label{mass_7}
\begin{split}
        \int_{C_A} \max\left\{\frac{\frac{(x-o')^T}{||x-o'||_2} (D(Id-T)(x)) \frac{x-o'}{||x-o'||_2}}{||x-o'||_2^{n-1}},0\right\}dx
    &\leq (2\ell \varepsilon^{-1})^{n-1}|C_A\setminus E|+c^{\ref{mass_6}}_{n,\varepsilon, \ell}\sqrt{\frac{\delta+\gamma}{t}}|A|\\
    &\leq (2\ell \varepsilon^{-1})^{n-1} 2\gamma|A|+ c^{\ref{mass_6}}_{n,\varepsilon, \ell}\sqrt{\frac{\delta+\gamma}{t}}|A|\leq c_{n,\varepsilon,\ell}^{\ref{mass_7}}\sqrt{\frac{\delta+\gamma}{t}}|A|.
\end{split}
\end{align}
Now, for a unit vector $y \in S^{n-1}(o,1)$ define $s_y:=\max \{s \colon o'+sy \in C_A\}.$ Define the function  $f_y(s)\colon [0,s_y] \rightarrow \mathbb{R}$, by $f_y(s):=\langle (o'+sy)-T(o'+sy), y  \rangle.$
It is easy to check that 
$\frac{d}{ds}(f_y)(s)=y^T \ D(Id-T)(o'+sy) \ y.$
Thus, 
by performing the change of variable $x(s,y)\colon \mathbb{R}\times S^{n-1}(o,1)\rightarrow \mathbb{R}^n,x(s,y)=o'+sy,$ we get
\begin{align}\label{mass_8}
\begin{split}
\int_{C_A}& \max\left\{\frac{\frac{(x-o')^T}{||x-o'||_2} (D(Id-T)(x)) \frac{x-o'}{||x-o'||_2}}{||x-o'||_2^{n-1}},0\right\}dx\\
&= \int_{S^{n-1}(o,1)}\int_{0\leq s\leq s_y} \max\left\{y^T \ D(Id-T)(o'+sy) \ y , 0 \right\} dsdy\\
&\geq \int_{S^{n-1}(o,1)}\max\left\{\int_{0\leq s\leq s_y} y^T \ D(Id-T)(o'+sy) \ y \ ds, 0 \right\} dy\\
&= \int_{S^{n-1}(o,1)}\max\left\{\int_{0\leq s\leq s_y} \frac{d}{ds}(f_y)(s) ds, 0 \right\} dy\\
&=\int_{S^{n-1}(o,1)} \max\left\{f_y(s_y)-f_y(0), 0\right\} dy\\
&=\int_{S^{n-1}(o,1)} \max\left\{ \langle (o'+s_yy)-T(o'+s_yy), y  \rangle - \langle o'-T(o'), y  \rangle, 0 \right\} dy
\end{split}
\end{align}
Combining \eqref{mass_7} and \eqref{mass_8}, it follows that 
\begin{align}\label{mass_9}
 \int_{S^{n-1}(o,1)} \max\left\{ \langle (o'+s_yy)-T(o'+s_yy), y  \rangle - \langle o'-T(o'), y  \rangle, 0 \right\} dy   \leq c_{n,\varepsilon,\ell}^{\ref{mass_7}}\sqrt{\frac{\delta+\gamma}{t}}|A|.
\end{align}
Integrating \eqref{mass_5} between $o$ and $o'$, and using that $o'\in B(o,2\ell)$ we find
\begin{align}\label{mass_10}
    \begin{split}
    \left|(T(o)-o)-(T(o')-o')\right|&=\left|\int_{0}^{|o-o'|} \left[D(T-Id)\left(o+t\frac{o'-o}{|o'-o|}\right)\right]  \frac{o'-o}{|o'-o|} dt\right|\\
    &\leq\int_{0}^{|o'-o|} \left|D(T-Id)\left(o+t\frac{o'-o}{|o'-o|}\right)\right|_{op}dt\\
    &\leq  |o-o'| c^{\ref{lem_ellipticregularity}}_{n, \varepsilon, \ell}\sqrt{\frac{\delta+\gamma}{t}}\leq  2\ell c^{\ref{lem_ellipticregularity}}_{n, \varepsilon, \ell}\sqrt{\frac{\delta+\gamma}{t}}
    \end{split}
\end{align}
Integrating this further over the unit sphere, we find
\begin{align}\label{mass_11}
\begin{split}\int_{S^{n-1}(o,1)} \max\left\{\left\langle(T(o)-o)-(T(o')-o'),y\right\rangle , 0 \right\} dy&\leq |S^{n-1}(o,1)|2\ell c^{\ref{lem_ellipticregularity}}_{n, \varepsilon, \ell}\sqrt{\frac{\delta+\gamma}{t}}\leq c^{\ref{mass_11}}_{n, \varepsilon, \ell}\sqrt{\frac{\delta+\gamma}{t}}|A|.\end{split} \end{align}
Hence, we can adjust \eqref{mass_9} to give
\begin{align}\label{mass_12}
\begin{split}
 \int_{S^{n-1}(o,1)} &\max\left\{ \langle (o'+s_yy)-T(o'+s_yy), y  \rangle - \langle o-T(o), y  \rangle, 0 \right\} dy   \\
 &\leq 
 \int_{S^{n-1}(o,1)} \max\left\{ \langle (o'+s_yy)-T(o'+s_yy), y  \rangle - \langle o'-T(o'), y  \rangle, 0 \right\} dy\\
 &+\int_{S^{n-1}(o,1)} \max\left\{\left\langle(T(o)-o)-(T(o')-o'),y\right\rangle , 0 \right\} dy\\
 &\leq  c_{n,\varepsilon,\ell}^{\ref{mass_7}}\sqrt{\frac{\delta+\gamma}{t}}|A|+c^{\ref{mass_11}}_{n, \varepsilon, \ell}\sqrt{\frac{\delta+\gamma}{t}}|A|= c_{n,\varepsilon,\ell}^{\ref{mass_12}}\sqrt{\frac{\delta+\gamma}{t}}|A|
 \end{split}
\end{align}

We aim to evaluate this as an integral over the boundary  $\partial C_A$ rather than the unit sphere $S^{n-1}(o,1)$. Recall that  
$o'\in (1-\varepsilon)C_A$ so that $(1-\varepsilon)o'+ \varepsilon C_A\subset C_A$. In particular, as 
$B(o,1/\ell)\subset C_A\subset B(o,\ell)$, we have $B(o', \varepsilon/\ell)\subset C_A\subset B(o',2\ell)$. Considering the map $z\colon S^{n-1}(o,1)\to \partial C_A; y\mapsto o'+s_yy$, so that $y=\frac{z-o'}{||z-o'||}$, then we find that the Jacobian of this map has determinant bounded by some constant, say $k_{n,\varepsilon,\ell}^{\ref{mass_13}}$, depending only on $\varepsilon,\ell,$ and $n$. Hence, changing variables, we find
\begin{align}\label{mass_13}
\begin{split}\int_{\partial C_A} &\max\left\{ \left\langle z-T(z), \frac{z-o'}{|z-o'|} \right\rangle - \left\langle o-T(o), \frac{z-o'}{|z-o'|}  \right\rangle, 0 \right\} dz   \\
&\leq k_{n,\varepsilon,\ell}^{\ref{mass_13}}\int_{S^{n-1}(o,1)} \max\left\{ \langle (o'+s_yy)-T(o'+s_yy), y  \rangle - \langle o-T(o), y  \rangle, 0 \right\} dy\\
&\leq   k_{n,\varepsilon,\ell}^{\ref{mass_13}}c_{n,\varepsilon,\ell}^{\ref{mass_12}}\sqrt{\frac{\delta+\gamma}{t}}|A|\leq c_{n,\varepsilon,\ell}^{\ref{mass_13}}\sqrt{\frac{\delta+\gamma}{t}}|A|.
 \end{split}
\end{align}
Note that $T^{-1}\colon C_B \rightarrow C_A$ is also an optimal transport map. By repeating the entire argument above, we get that for $o^\star \in (1-\varepsilon)C_B$ 
\begin{align*}
\begin{split}\int_{\partial C_B} &\max\left\{ \left\langle w-T^{-1}(w), \frac{w-o^\star}{|w-o^\star|} \right\rangle - \left\langle o-T^{-1}(o), \frac{w-o^\star}{|w-o^\star|}  \right\rangle, 0 \right\} dw   \leq  c_{n,\varepsilon,\ell}^{\ref{mass_13}}\sqrt{\frac{\delta+\gamma}{t}}|A|.
 \end{split}
\end{align*}
We now observe that $T^{-1}(o)$ belongs to $(1-\varepsilon)C_A$. Indeed,  Lemma \ref{lem_ellipticregularity} applied to $T^{-1}$ implies that $T^{-1}$ is uniformly close to the affine map $x+T^{-1}(o)$ inside $(1-\varepsilon)C_B$. Since $T^{-1}((1-\varepsilon)C_B)\subset C_A$, this implies that $T^{-1}(o)$ remains at some uniform positive distance from $\partial C_A$.

Now, integrating \eqref{mass_5} between $o$ and $T^{-1}(o)$ (both of which are in $(1-\varepsilon)C_A$), we get
\begin{align*}
    \begin{split}
    \left|(T(o)-o)-(o-T^{-1}(o))\right|&=\left|(T(o)-o)-(T(T^{-1}(o))-T^{-1}(o))\right|\\
    &=\left|\int_{0}^{|o-T^{-1}(o)|} \left[D(T-Id)\left(o+s\frac{T^{-1}(o)-o}{|T^{-1}(o)-o|}\right)\right]  \frac{T^{-1}(o)-o}{|T^{-1}(o)-o|} ds\right|\\
    &\leq\int_{0}^{|o-T^{-1}(o)|} \left|D(T-Id)\left(o+s\frac{T^{-1}(o)-o}{|T^{-1}(o)-o|}\right)\right|_{op}ds\\
    &\leq  |o-T^{-1}(o)| c^{\ref{lem_ellipticregularity}}_{n, \varepsilon, \ell}\sqrt{\frac{\delta+\gamma}{t}}\leq  2\ell c^{\ref{lem_ellipticregularity}}_{n, \varepsilon, \ell}\sqrt{\frac{\delta+\gamma}{t}}
    \end{split}
\end{align*}
Combining the last two equations and using the fact that $|\partial C_B| \leq |S^{n-1}(o, \ell)|$, we get that
\begin{align}\label{mass_14}
\begin{split}\int_{\partial C_B} &\max\left\{ \left\langle w-T^{-1}(w), \frac{w-o^\star}{|w-o^\star|} \right\rangle - \left\langle T(o)-o, \frac{w-o^\star}{|w-o^\star|}  \right\rangle, 0 \right\} dw\\  
&\leq c_{n,\varepsilon,\ell}^{\ref{mass_13}}\sqrt{\frac{\delta+\gamma}{t}}|A| + |S^{n-1}(o, \ell)|2\ell c^{\ref{lem_ellipticregularity}}_{n, \varepsilon, \ell}\sqrt{\frac{\delta+\gamma}{t}} \leq  c_{n,\varepsilon,\ell}^{\ref{mass_14}}\sqrt{\frac{\delta+\gamma}{t}}|A|.
 \end{split}
\end{align}

We apply \Cref{prop_conelike_0} to the $(\gamma,\ell,\lambda,\mu)$ conelike sets $A,B$, together with the vector $y_2=o-T(o)$ and the map $M=T$ in the case $s=1$ and the map $M=T^{-1}$ in the case $s=-1$ (restricted to $\partial C_A \cap \partial C_B$). Thus, we find faces $F_A$ of $C_A$ and $F_B$ of $C_B$ with the same supporting hyperplane $H$, and we find $w_0 \in H$ such that  $B^n(w_0,1/r^{\ref{prop_conelike_0}}_{n, \lambda, \ell}) \cap H \subset F_A \cap F_B.$
Moreover, for every $w \in B^n(w_0,1/r^{\ref{prop_conelike_0}}_{n, \lambda, \ell}) \cap H $ there exists a ball $X_w \subset \mathbb{R}^n$ such that with $y_1= M(w)-w$ we have
\begin{enumerate}
    \item  $X_w \subset (1-\varepsilon^{\ref{prop_conelike_0}}_{n, \lambda, \ell})(C_A\cap C_B)$
    \item $|X_w| \geq  m^{\ref{prop_conelike_0}}_{n, \lambda, \ell}$
    \item $d(w, X_w) \geq 1/(4r^{\ref{prop_conelike_0}}_{n \lambda, \ell}) $
    \item $\mathbb{P}_{x \in X_w} \left(\frac{\langle y_1, x-w \rangle}{|y_1||x-w|} \geq 0\right)\geq 1/2$
    \item $\mathbb{P}_{x \in X_w} \left(\frac{\langle sy_2, x-w \rangle}{|y_2||x-w|} \geq \sigma^{\ref{prop_conelike_0}}_{n, \lambda, \ell}\right) =1$.
\end{enumerate}
Consider the case $s=1$ and $M=T$ (the other case is analogous). By averaging \eqref{mass_13} over $o' \in (1-\varepsilon)C_A$, we get
$$ \mathbb{E}_{o' \in (1-\varepsilon)C_A}\int_{\partial C_A} \max\left\{ \left\langle z-T(z), \frac{z-o'}{|z-o'|} \right\rangle - \left\langle o-T(o), \frac{z-o'}{|z-o'|}  \right\rangle, 0 \right\} dz \leq c_{n,\varepsilon,\ell}^{\ref{mass_13}}\sqrt{\frac{\delta+\gamma}{t}}|A|,$$
which, by interchanging the integration and the average, is equivalent to
$$ \int_{\partial C_A} \mathbb{E}_{o' \in (1-\varepsilon)C_A} \max\left\{ \left\langle z-T(z), \frac{z-o'}{|z-o'|} \right\rangle - \left\langle o-T(o), \frac{z-o'}{|z-o'|}  \right\rangle, 0 \right\} dz \leq c_{n,\varepsilon,\ell}^{\ref{mass_13}}\sqrt{\frac{\delta+\gamma}{t}}|A|.$$

By restricting our attention to a certain part of the boundary, namely $B^n(w_0,1/r^{\ref{prop_conelike_0}}_{n, \lambda, \ell}) \cap H  \subset \partial C_A$, we deduce
$$\int_{B^n\left(w_0,1/r^{\ref{prop_conelike_0}}_{n, \lambda, \ell}\right) \cap  H} \mathbb{E}_{o' \in (1-\varepsilon)C_A} \max\left\{ \left\langle z-T(z), \frac{z-o'}{|z-o'|} \right\rangle - \left\langle o-T(o), \frac{z-o'}{|z-o'|}  \right\rangle, 0 \right\} dz \leq c_{n,\varepsilon,\ell}^{\ref{mass_13}}\sqrt{\frac{\delta+\gamma}{t}}|A|.$$

For each $z\in B^n(w_0,1/r^{\ref{prop_conelike_0}}_{n, \lambda, \ell}) \cap  H$, by conditioning on the event $o' \in X_z$ and using the first two properties of $X_z$, namely that $X_z \subset (1-\varepsilon^{\ref{prop_conelike_0}}_{n, \lambda, \ell}) C_A \subset (1-\varepsilon) C_A,$ and that $|X_z| \geq m^{\ref{prop_conelike_0}}_{n, \lambda, \ell}$  we get
$$\int_{B^n(w_0,1/r^{\ref{prop_conelike_0}}_{n, \lambda, \ell}) \cap  H} \mathbb{E}_{o' \in X_z} \max\left\{ \left\langle z-T(z), \frac{z-o'}{|z-o'|} \right\rangle - \left\langle o-T(o), \frac{z-o'}{|z-o'|}  \right\rangle, 0 \right\} dz \leq (m^{\ref{prop_conelike_0}}_{n, \lambda, \ell})^{-1} c_{n,\varepsilon,\ell}^{\ref{mass_13}}\sqrt{\frac{\delta+\gamma}{t}}|A|.$$

Now for each  $z\in B^n(w_0,1/r^{\ref{prop_conelike_0}}_{n, \lambda, \ell}) \cap  H$, by conditioning on  the event $o'\in E_z$, where $$E_z :=\left\{o' \in X_z \colon \frac{\langle y_1, o'-z \rangle}{|y_1||o'-z|} \geq 0 \text{ and } \frac{\langle y_2, o'-z \rangle}{|y_2||o'-z|} \geq \sigma^{\ref{prop_conelike_0}}_{n, \lambda, \ell} \right\} $$
and using the last two properties of $X_z$ which imply $|E_z| \geq \frac12|X_z|$, we get
$$\int_{B^n\left(w_0,1/r^{\ref{prop_conelike_0}}_{n, \lambda, \ell}\right) \cap  H} \mathbb{E}_{o' \in E_z} \max\left\{ \left\langle z-T(z), \frac{z-o'}{|z-o'|} \right\rangle - \left\langle o-T(o), \frac{z-o'}{|z-o'|}  \right\rangle, 0 \right\} dz \leq 2 (m^{\ref{prop_conelike_0}}_{n, \lambda, \ell})^{-1} c_{n,\varepsilon,\ell}^{\ref{mass_13}}\sqrt{\frac{\delta+\gamma}{t}}|A|.$$

For each $z\in B^n(w_0,1/r^{\ref{prop_conelike_0}}_{n, \lambda, \ell}) \cap  H$ and each $o' \in E_z$, by the definition of $E_z$, $y_1$ and $y_2$, we have
$\left\langle z-T(z), \frac{z-o'}{|z-o'|} \right\rangle \geq 0 $
and
$ -\left\langle o-T(o), \frac{z-o'}{|z-o'|}  \right\rangle \geq \sigma^{\ref{prop_conelike_0}}_{n, \lambda, \ell}  |o-T(o)|.$

By combining the last three inequalities, we obtain
$$\left|B^n(w_0,1/r^{\ref{prop_conelike_0}}_{n, \lambda, \ell}) \cap  H\right| \sigma^{\ref{prop_conelike_0}}_{n, \lambda, \ell}  |o-T(o)| \leq 2 (m^{\ref{prop_conelike_0}}_{n, \lambda, \ell})^{-1} c_{n,\varepsilon,\ell}^{\ref{mass_13}}\sqrt{\frac{\delta+\gamma}{t}}|A|, $$
hence
\begin{align}\label{mass_15}
\begin{split}
    |T(o)-o| &\leq |B^n(w_0,1/r^{\ref{prop_conelike_0}}_{n, \lambda, \ell}) \cap  H|^{-1} (\sigma^{\ref{prop_conelike_0}}_{n, \lambda, \ell} )^{-1}2 (m^{\ref{prop_conelike_0}}_{n, \lambda, \ell})^{-1} c_{n,\varepsilon,\ell}^{\ref{mass_13}}\sqrt{\frac{\delta+\gamma}{t}}|A|\leq c^{\ref{mass_15}}_{n, \lambda, \varepsilon, \ell}\sqrt{\frac{\delta+\gamma}{t}}.
\end{split}
\end{align}

By combining \eqref{mass_15} with \eqref{mass_13}, we conclude
\begin{align*}
    \begin{split}
        \int_{\partial C_A} &\max\left\{ \left\langle z-T(z), \frac{z-o'}{|z-o'|} \right\rangle , 0 \right\} dz    \\
       &\leq \int_{\partial C_A} \max\left\{ \left\langle z-T(z), \frac{z-o'}{|z-o'|} \right\rangle - c^{\ref{mass_15}}_{n, \lambda, \varepsilon, \ell}\sqrt{\frac{\delta+\gamma}{t}}, 0 \right\} + c^{\ref{mass_15}}_{n, \lambda, \varepsilon, \ell}\sqrt{\frac{\delta+\gamma}{t}} dz     \\
        &\leq|\partial C_A|c^{\ref{mass_15}}_{n, \lambda, \varepsilon, \ell}\sqrt{\frac{\delta+\gamma}{t}}+ \int_{\partial C_A} \max\left\{ \left\langle z-T(z), \frac{z-o'}{|z-o'|} \right\rangle - \left\langle o-T(o), \frac{z-o'}{|z-o'|}  \right\rangle, 0 \right\} dz   \\
        &\leq |\partial C_A|c^{\ref{mass_15}}_{n, \lambda, \varepsilon, \ell}\sqrt{\frac{\delta+\gamma}{t}}+c_{n,\varepsilon,\ell}^{\ref{mass_13}}\sqrt{\frac{\delta+\gamma}{t}}|A|\leq
         c_{n,\lambda, \varepsilon, \ell}\sqrt{\frac{\delta+\gamma}{t}}|A|.
    \end{split}
\end{align*}

\end{proof}

\subsubsection{Proof of \Cref{probabilisticlem}}

\begin{proof}[Proof of \Cref{probabilisticlem}]
First note that if $x\in C_B$, then $d(x,C_B)=0$, so the inequality trivially holds. Henceforth assume $x\not\in C_B$.
Define
\begin{align*}
\psi&:=0.1, \quad \phi=\frac{1}{4\ell},\quad
\xi:=\min\left\{\frac{1}{12}\phi\ell^{-1},\frac12\psi\ell^{-1}\right\},\quad
\theta:=2\xi^{-2}\ell^2,\quad
\zeta:=\frac{1}{24}\psi\theta^{-2}\ell^{-2},\\
\alpha&:=\min\left\{ \frac14\xi^2\theta^{-2}\ell^{-2}(n-1)^{-1}, \frac{1}{48^2}\psi^2n\theta^{-6}\ell^{-6}, \frac12\right\}, \text{ and }
\eta:=\min\left\{\frac13\psi \theta^{-1} \ell^{-1}, \frac12\phi\right\}
\end{align*}

Write $e_1,\dots, e_n\in\mathbb{R}^n$ and $\theta_1,\dots, \theta_n\in[\theta^{-1},\theta]$ for the random parameters corresponding to transformation $Q\sim \mathcal{Q}_\theta$.

We first restrict our attention to a controlled set of transformations $Q$.
We condition on the event that $\theta_1\leq \theta^{-1}\min_{i>1}\{\theta_i\}$ and the event that $e_1$ points roughly in the direction $x$, viz $\left\langle \frac{x}{|x|},e_1\right\rangle\geq 1-\alpha$. As these events are independent, there exists a constant $c_n^{\ref{aimingatthecentre}}$ so that
\begin{align}\label{aimingatthecentre}
\PP\left(\theta_1\leq \theta^{-1}\min_{i>1}\{\theta_i\}\text{ and }\left\langle \frac{x}{|x|},e_1\right\rangle\geq 1-\alpha\right)\geq c_n^{\ref{aimingatthecentre}}.
\end{align}
Henceforth, we condition on these events. We will show that, for these $Q$, the stated inequality holds. For notational convenience, rescale by a factor $\theta^{-1}/\theta_1$, so that we may assume that $\theta_1=\theta^{-1}$ and $\theta_{2}, \dots, \theta_n \in [1,\theta]$.

First, note that as $\langle x,e_1\rangle\geq (1-\alpha)|x|$, we have 
$\langle x,e_i\rangle\leq \sqrt{1-(1-\alpha)^2}|x|\leq \sqrt{2\alpha}\ell,$
which implies that 
\begin{align}\label{smallQx}
|Q(x)|=\sqrt{\sum_i \langle Q(x),e_i\rangle^2}\leq \sqrt{ \theta^{-1}\langle x,e_1\rangle^2+\sum_{i>1} \theta^2\langle x,e_i\rangle^2} \leq \sqrt{ \theta^{-1}\ell^2+(n-1) 2\alpha\theta^2\ell^2 }\leq \xi,
\end{align}
\begin{align}\label{bigQx}
|Q(x)|\geq \langle Q(x),e_1\rangle= \theta^{-1}\langle x,e_1\rangle> 0.9\theta^{-1} \ell^{-1}. 
\end{align}

Let $u:=\frac{Q(x)-T_Q(Q(x))}{|Q(x)-T_Q(Q(x))|}$. We show that $\langle u,e_1\rangle$ is not very negative.

\begin{clm}
$\langle u,e_1\rangle\geq -\psi$.
\end{clm}
\begin{proof}[Proof of claim]
Assume for a contradiction $\langle u,e_1\rangle<-\psi$. Let $p$ be the point where the line through $Q(x)$ and $T_Q(Q(x))$ intersects the plane spanned by $e_2,\dots, e_n$. Write $p-Q(x)=s u$ for some $s \in \mathbb{R}$. Note $s>0$ as $\langle Q(x),e_1\rangle>0$ and $\langle u,e_1\rangle<0$. Since 
$\langle Q(x),e_1\rangle\leq |Q(x)|\leq \xi$ and $\langle u,e_1\rangle<-\psi$,
we find that $s\leq \xi/\psi$. By the triangle inequality, this implies 
$|p|\leq |p-Q(x)|+|Q(x)|\leq \xi+ \xi/\psi<1/\ell.$
Since, $\theta_i\geq 1$ for all $i>1$, we have
$B^n(o,1/\ell)\cap \text{span}(e_2,\dots,e_n)\subset Q(B^n(o,1/\ell))\cap \text{span}(e_2,\dots,e_n),$
so that $|p|\leq 1/\ell$ implies $p\in Q(B^n(o,1/\ell))$. Moreover, $Q(B^n(o,1/\ell))\subset Q(C_B)$, so $p\in Q(C_B)$.
However, this implies $Q(x)$ lies on the line segment between $p$ and $T_Q(Q(x))$, both of which are in $Q(C_B)$. Since affine transformations preserve convexity, this implies $Q(x)\in Q(C_B)$, i.e., $x\in C_B$, a contradiction.
\end{proof}

Let us return to the inner product
$
 \left\langle Q(x)-T_Q(Q(x)),\frac{Q(x)-Q(o')}{||Q(x)-Q(o')||_2}\right\rangle= |Q(x)-T_Q(Q(x))| \left\langle u,\frac{Q(x)-o'}{||Q(x)-o'||_2}\right\rangle,
$
for some $o'\in Q(B(o,1/\ell))$. Write
$$\mathcal{O}:=\left\{o'\in Q(B(o,1/\ell)):  \left\langle u,\frac{Q(x)-o'}{||Q(x)-o'||_2}\right\rangle\geq \eta\right\}.$$
We shall argue $|\mathcal{O}| \geq c_n^{\ref{o'boundeq}} |Q(B(o,1/\ell))|$.
Write $\pi$ for the projection onto the plane spanned by $e_2,\dots, e_n$, thus 
$$ \left\langle u,\frac{Q(x)-o'}{||Q(x)-o'||_2}\right\rangle =\left\langle \pi(u),\pi\left(\frac{Q(x)-o'}{||Q(x)-o'||_2}\right)\right\rangle +\left(\left\langle u,e_1\right\rangle\cdot \left\langle \frac{Q(x)-o'}{||Q(x)-o'||_2},e_1\right\rangle\right),$$
and distinguish two cases; either $\langle u,e_1\rangle\geq \psi$ or $\langle u,e_1\rangle\in [-\psi, \psi]$. 

In the former case, consider the set 
$$\mathcal{O}':=\{o'\in Q(B(o,1/\ell)): \langle o',e_1\rangle\leq 0, ||o'||\leq \zeta \}. $$
Note that as $\zeta<\theta^{-1}\ell^{-1}$, we have that 
$\{o'\in Q(B(o,1/\ell)): ||o'||\leq \zeta \}=B(o,\zeta),$
so that using symmetry in the plane spanned by $e_2,\dots, e_n$ we have $|\mathcal{O}'|=\frac12 |B(o,\zeta)| \geq \frac{\ell^n \zeta^{n}}{2\theta^{n-1}}|Q(B(o,1/\ell))|$.

For points $o'\in\mathcal{O}'$, using \Cref{bigQx} and a version of \Cref{smallQx}, we get 
\begin{align*}\left|\left\langle \pi(u),\pi\left(\frac{Q(x)-o'}{||Q(x)-o'||_2}\right)\right\rangle\right|&\leq\left|\pi\left(\frac{Q(x)-o'}{||Q(x)-o'||_2}\right)\right|\leq \frac{\left|\pi\left(Q(x)\right)\right|+\left|\pi\left(o'\right)\right|}{0.9||Q(x)||_2}\leq \frac{2\sqrt{(n-1)2\alpha}\theta\ell}{\theta^{-1}\ell^{-1}}+\frac{2\zeta}{\theta^{-1}\ell^{-1}}.
\end{align*}
On the other hand, because $\langle o', e_1 \rangle \leq 0$ we have 
\begin{align*}
\left\langle u,e_1\right\rangle\cdot \left\langle \frac{Q(x)-o'}{||Q(x)-o'||_2},e_1\right\rangle&\geq \psi \left\langle \frac{Q(x)}{||Q(x)-o'||_2},e_1\right\rangle\geq \psi \theta^{-1}\left\langle \frac{x}{||Q(x)||+||o'||},e_1\right\rangle\\
&\geq \frac{\psi \theta^{-1}}{\xi+\zeta}\left\langle x,e_1\right\rangle\geq \frac{\psi \theta^{-1}}{\xi+\zeta}(1-\alpha)\ell^{-1}\\
\end{align*}

Combining these two bounds we find
\begin{align*}
\left\langle u,\frac{Q(x)-o'}{||Q(x)-o'||_2}\right\rangle &= \left(\left\langle u,e_1\right\rangle\cdot \left\langle \frac{Q(x)-o'}{||Q(x)-o'||_2},e_1\right\rangle\right)+\left\langle \pi(u),\pi\left(\frac{Q(x)-o'}{||Q(x)-o'||_2}\right)\right\rangle \\
&\geq \frac{\psi \theta^{-1}}{\xi+\zeta}(1-\alpha)\ell^{-1}-\frac{2\sqrt{(n-1)2\alpha}\theta\ell}{\theta^{-1}\ell^{-1}}-\frac{2\zeta}{\theta^{-1}\ell^{-1}}\\
&=\theta^{-1}\ell^{-1}\left(\frac{\psi}{2}-4\sqrt{n\alpha}\theta^3\ell^3-2\zeta\theta^2\ell^2\right)\geq \frac{\psi}{3\theta \ell}\geq \eta.
\end{align*}
Hence, we find $\mathcal{O}'\subset\mathcal{O}$, so that $|\mathcal{O}|\geq |\mathcal{O}'|\geq \frac12 |B(o,\zeta)| \geq \frac{\ell^n \zeta^{n}}{2\theta^{n-1}}|Q(B(o,1/\ell))|$.

Now consider the other case, i.e., $\langle u,e_1\rangle\in [-\psi, \psi]$. This implies that $|\pi(u)|\geq \sqrt{1-\psi^2}\geq \frac12$. Write $u':=\pi(u)/|\pi(u)|$. We consider the set 
$$\mathcal{O}'':=\{o'\in Q(B(o,1/\ell)): \langle o',e_1\rangle\geq 0, ||\pi(o')||\in (1/2\ell,1/\ell), \langle \pi(o'),u'\rangle < -\phi\}. $$
By symmetry in the plane spanned by $e_2,\dots, e_n$, we have that
\begin{align*}
|\mathcal{O}''|&=\frac12 \left|\{o'\in Q(B(o,1/\ell)): ||\pi(o')||\in (1/2\ell,1/\ell), \langle \pi(o'),u'\rangle < -\phi\}\right|
\end{align*}
Consider the transformation $Q'\sim \mathcal{Q}_\theta$ with parameters  $e_1,\dots, e_n\in\mathbb{R}^n$ (same as Q) and also $\theta^{-1}, 1, \dots, 1$. As $\theta_1=\theta^{-1}$, $\theta_2, \dots, \theta_n \in [1,\theta]$, we get $Q'(B(o, 1/\ell)) \subset Q(B(o, 1/\ell))$ and $\frac{|Q'(B(o, 1/\ell))|}{ |Q(B(o, 1/\ell))|} \geq \theta^{-n+1}$. From this containment and the rotational symmetry of $Q'(B(o, 1/\ell))$  around the $e_1$ axis, we deduce

\begin{align*} |O''| &\geq \frac12 \left|\{o'\in Q'(B(o,1/\ell)): ||\pi(o')||\in (1/2\ell,1/\ell), \langle \pi(o'),u'\rangle < -\phi\}\right|\\
&\geq \frac12 \left|\{o'\in Q'(B(o,1/\ell)): ||\pi(o')||\in (1/2\ell,1/\ell), \langle \frac{\pi(o')}{||\pi(o')||},u'\rangle < -2\ell \phi\}\right|\\
&\geq \frac12 \frac{\cos^{(-1)}(2\ell \phi)}{2\pi} \left|\{o'\in Q'(B(o,1/\ell)): ||\pi(o')||\in (1/2\ell,1/\ell)\}\right| \\
&\geq \frac16 \left|\{o'\in Q'(B(o,1/\ell)): ||\pi(o')||\in (1/2\ell,1/\ell)\}\right| \\
&\geq \frac1{12} \left|\{o'\in Q'(B(o,1/\ell))\}\right| \geq \frac{1}{12 \theta^{n-1}} \left|\{o'\in Q(B(o,1/\ell))\}\right|.  
\end{align*}
Assume that $o'\in\mathcal{O}''$. We have
\begin{align*}
\left|\left\langle u,e_1\right\rangle\cdot \left\langle \frac{Q(x)-o'}{||Q(x)-o'||_2},e_1\right\rangle\right|&\leq \psi \left|\left\langle \frac{Q(x)}{||Q(x)-o'||_2},e_1\right\rangle\right|\leq \psi \frac{||Q(x)||}{||o'||-|Q(x)|}\leq 3\ell\psi \xi,
\end{align*}
where the first inequality follows from $\langle o',e_1\rangle\geq 0$, the second inequality follows from the triangle inequality and $|e_1|=1$ and the final inequality follows from $|Q(x)|\leq \xi$ and $||o'||-|Q(x)|\geq 1/2\ell-\xi\geq 1/3\ell$.

For the other term, we use $|Q(x)|\leq \xi$ and 
$|o'|\leq \sqrt{ |\pi(o')|^2+ \langle o',e_1\rangle^2 }\leq \sqrt{1/\ell^2 + 1/\ell^2}\leq 2/\ell$
to find that
\begin{align*}
\left\langle \pi(u),\pi\left(\frac{Q(x)-o'}{||Q(x)-o'||_2}\right)\right\rangle &\geq \frac{|\pi(u)|}{||o'||_2+|Q(x)|}\left(\left\langle u',\pi(-o')\right\rangle-|\langle \pi(Q(x)),u'\rangle|\right)\geq \frac{1/2}{2/\ell+\xi}\left(\left\langle u',\pi(-o')\right\rangle-\xi\right)\geq \phi-\xi.
\end{align*}
Combining these two inequalities, we find
\begin{align*}
\left\langle u,\frac{Q(x)-o'}{||Q(x)-o'||_2}\right\rangle &= \left(\left\langle u,e_1\right\rangle\cdot \left\langle \frac{Q(x)-o'}{||Q(x)-o'||_2},e_1\right\rangle\right)+\left\langle \pi(u),\pi\left(\frac{Q(x)-o'}{||Q(x)-o'||_2}\right)\right\rangle \geq \phi-\xi- 3\ell\psi\xi\geq \phi/2 \geq \eta
\end{align*}
This proves that $\mathcal{O''}\subset\mathcal{O}$, hence $|\mathcal{O''}|\leq|\mathcal{O}|$.

Returning to the two cases $\langle u,e_1\rangle\geq \psi$ and $\langle u,e_1\rangle\in [-\psi, \psi]$, we now find that in both cases \begin{align}\label{o'boundeq}|\mathcal{O}|\geq \min\{ |\mathcal{O}'|,|\mathcal{O}''|\}\geq c_n^{\ref{o'boundeq}} |Q(B(o,1/\ell))|,\end{align}
where $c_n^{\ref{o'boundeq}}>0$ can be found in terms of $\zeta, \theta,\ell$, and $n$.
Note that if $o'\in\mathcal{O}$, then 
\begin{align*}
 \left\langle Q(x)-T_Q(Q(x)),\frac{Q(x)-o'}{||Q(x)-o'||_2}\right\rangle&= |Q(x)-T_Q(Q(x))| \left\langle u,\frac{Q(x)-o'}{||Q(x)-o'||_2}\right\rangle\geq \eta |Q(x)-T_Q(Q(x))|\\
 &\geq \eta \theta^{-1} |x-Q^{-1}(T_Q(Q(x)))|\geq \eta \theta^{-1} d(x,C_B),
\end{align*}
where the first inequality follows from the definition of $\mathcal{O}$, the second inequality follows from $|Q^{-1}|_{op}\leq \theta$ and the last inequality follows from the fact that $Q^{-1}(T_Q(Q(x)))\in C_B$.

Now we are ready to conclude using the following Markov bound on the expectation we are trying to control:
\begin{align*}
\mathbb{E}_{Q,o'}&\left[ \max\left\{ \left\langle Q(x)-T_Q(Q(x)),\frac{Q(x)-Q(o')}{||Q(x)-Q(o')||_2}\right\rangle, 0 \right\} \right]\\
    &\geq \PP\left(\theta_1\leq \theta^{-1}\min_{i>1}\{\theta_i\}\text{ and }\left\langle \frac{x}{|x|},e_1\right\rangle\geq 1-\alpha\right) \PP\left(Q(o')\in \mathcal{O}|Q\right) \eta \theta^{-1}d(x,C_B) \\
    &\geq c_n^{\ref{aimingatthecentre}} c_n^{\ref{o'boundeq}} \eta \theta^{-1} d(x,C_B)\geq c^{\ref{probabilisticlem}}_n d(x,C_B).
\end{align*}
Here we used that if $o'$ is chosen uniformly from $B(o,1/\ell)$, then $Q(o')$ is chosen uniformly from $Q(B(o,1/\ell))$. This concludes the lemma.
\end{proof}

\subsubsection{Proof of \Cref{sdvsdistanceint}}

\begin{proof}[Proof of \Cref{sdvsdistanceint}]
First note that $|X\triangle Y|=2|X\setminus Y|$, so it suffices to bound $|X\setminus Y|$.

Given $x\in(\partial X )\setminus Y$, let $y_x$ be the intersection between the line segment $ox$ and $\partial Y$. We'll show that $|x-y_x|=O_{\ell}(d(x,Y))$ and integrate $|x-y_x|$ over $x$ to find the lemma.

\begin{clm}
 $|x-y_x|\leq \ell^2d(x,Y)$.
\end{clm}
\begin{proof}[Proof of claim]
Let $p_x$ be the projection of $x$ onto $\partial Y$, so that $d(x,Y)=|x-p_x|$. Note that as $x,y_x,$ and $o$ are colinear, $x,y_x,p_x$ and $o$ are coplanar. Restrict attention to this plane, and let $L$ be the ray (line) through $p_x$ tangent to $B(o,1/\ell)$ so that $L$ intersects the line segment $ox$. Write $y'_x$ for that intersection. Note that because $p_x\in Y$ and $B(o,1/\ell)\subset Y$, we have $|x-y'_x|\geq |x-y_x|$, so it suffices to upper bound $|x-y'_x|$. We  show that the angle $\angle L, ox$ is lower bounded away from $0$ in terms of $\ell$.

Let $t$ be the tangent point of $L$ to $B(0,1/\ell)$, so that $\angle L, ox=\angle ty'_xo$. Using the sin rule in the triangle $ty'_xo$, we find
$\frac{\sin(\angle ty'_xo)}{|t-o|}=\frac{\sin(\angle y'_xto)}{|y'_x-o|}= \frac{1}{|y'_x-o|},$
so that using $|y'_x-o|\leq \ell$ and $|t-o|=1/\ell$, we find $\sin(\angle ty'_xo)\geq \ell^{-2}$.
Considering the triangle $y'_xp_xx$, we find $\angle p_xy'_xx=\angle ty'_xo$, so that applying the sin rule again, we find
$|y'_x-x|=\frac{\sin(\angle xp_xy'_x)}{\sin(\angle p_xy'_xx)} |x-p_x|\leq \ell^2 |x-p_x|.$
We conclude
$$|x-y_x|\leq |y'_x-x|\leq \ell^2 |x-p_x|=\ell^2 d(x,Y).$$
\end{proof}
Using this claim, we find
$$\int_{\partial X}|x-y_x|dx\leq \ell^2\int_{\partial X}d(x,Y)dx.$$
Note that
$\bigcup_{x\in\partial X} [x,y_x]= X\setminus Y.$
Let $z\colon S^{n-1}(o,\ell)\to \partial X$ be the map taking a direction $v\in S^{n-1}(o,\ell)$ to the intersection between $\mathbb{R}^+v$ and $\partial X$.
Note that 
\begin{align*}
\begin{split}
      \left|\bigcup_{x\in\partial X} [x,y_x]\right|&\leq \left|\bigcup_{v\in\partial S^{n-1}(o,\ell)} [v-(z(v)-y_{z(v)}),v]\right|=  \int_{0 \leq s \leq \ell}  \left| S^{n-1}(o,s) \bigcap \bigcup_{v\in\partial S^{n-1}(o,\ell)} [v-(z(v)-y_{z(v)}),v]\right| ds\\
      &=  \int_{0 \leq s \leq \ell} \frac{|S^{n-1}(o,s)|}{|S^{n-1}(o,\ell)|} \left| \{v \in S^{n-1}(o,\ell) \colon |z(v)-y_{z(v)}| \geq \ell-s  \} \right| ds\\
      &\leq\int_{0 \leq s \leq \ell} \left| \{v \in S^{n-1}(o,\ell) \colon |z(v)-y_{z(v)}| \geq \ell-s  \} \right| ds= \int_{v\in\partial S^{n-1}(o,\ell)}|z(v)-y_{z(v)}|dv
\end{split}
\end{align*}
The first inequality is immediate from the fact that we compress segments inside $B^n(o,\ell)$ radially outwards onto the sphere $S^{n-1}(o,\ell)$. 

As $B(o,1/\ell)\subset X\subset B(o,\ell)$, we find that the Jacobian of the map $z$ has determinant bounded by some constant, say $k_{n,\ell}$, depending only on $\ell,$ and $n$. Hence, we find 
\begin{align*}
  |X\setminus Y|\leq \int_{v\in\partial S^{n-1}(o,\ell)}|z(v)-y_{z(v)}|dv  \leq k_{n,\ell}\int_{\partial X}|x-y_x|dx\leq k_{n,\ell}\ell^2\int_{\partial X} d(x,Y)dx,
\end{align*}
which concludes the proof.
\end{proof}

\subsection{Proofs of Lemmas}

\subsubsection{Proof of \Cref{prop_conelike_0}}

\begin{proof}[Proof of \Cref{prop_conelike_0}]
Fix $r^{\ref{prop_conelike_0}}_{n \lambda, \ell}=2r^{\ref{lem_conelike_1}}_{n \lambda, \ell}$, $\sigma^{\ref{prop_conelike_0}}_{n, \lambda, \ell}=\sigma^{\ref{lem_conelike_3}}_{n,r^{\ref{lem_conelike_1}}_{n \lambda, \ell}}/4$, $m^{\ref{prop_conelike_0}}_{n, \lambda, \ell}=\left( k^{\ref{lem_conelike_final}}_{n,2r^{\ref{lem_conelike_1}}_{n \lambda, \ell},\sigma^{\ref{lem_conelike_3}}_{n,r^{\ref{lem_conelike_1}}_{n \lambda, \ell}}}\right)^n|B^n(o,1)| $, $\varepsilon^{\ref{prop_conelike_0}}_{n, \lambda, \ell}=1/(32(r^{\ref{lem_conelike_1}}_{n \lambda, \ell})^2)$.

Recall \Cref{conelike} and construct simplex $S'=S''+z$ with a vertex at $z$. Let $F_0, F_1, \dots F_n$ be the faces of $S'$ where $F_0$ is the face opposite vertex $z$. Then, by \Cref{lem_conelike_1},
\begin{enumerate}
    \item $ S' \subset C_A \cap C_B$
    \item $F_1 \cup \dots \cup F_n \subset \partial C_A \cap \partial C_B$.
    \item $B(u, 1/r^{\ref{lem_conelike_1}}_{n \lambda, \ell}) \subset S' \subset B(u, r^{\ref{lem_conelike_1}}_{n \lambda, \ell}) $ for some $u \in \mathbb{R}^n$.
\end{enumerate}

Let $f_0, f_1, \hdots, f_n$ be the inward normal vectors to the faces $F_0, \dots, F_n$, respectively. By \Cref{lem_conelike_3} together with (3), there exists $1 \leq i \leq n$ such that $  \frac{|\langle f_i, y_2  \rangle|}{|f_i||y_2|} \geq \sigma^{\ref{lem_conelike_3}}_{n,r^{\ref{lem_conelike_1}}_{n \lambda, \ell}}.$
Hence there exists $s \in \{\pm 1\}$ such that
\begin{align}\label{prop_cone_1}
    \frac{\langle f_i, sy_2  \rangle}{|f_i||sy_2|} \geq \sigma^{\ref{lem_conelike_3}}_{n,r^{\ref{lem_conelike_1}}_{n \lambda, \ell}}.
\end{align} 

Write $F=F_i$ and $f=f_i$. Let $H$ be the supporting hyperplane of $F$ and let $H^+$ and $H^-$ be the partition into half-spaces determined by $H$ with $H^+$ containing $S'$ and $H^-$ disjoint from the interior of $S'$.  

By \Cref{lem_conelike_1.5}, together with (3), we deduce there exists $w_0 \in F$ such that
$B^n(w_0,1/r^{\ref{lem_conelike_1}}_{n, \lambda, \ell})\cap H^+ \subset S'$
and
$B^n(w_0, 1/r^{\ref{lem_conelike_1}}_{n, \lambda, \ell}) \cap H \subset F.$

By (1), $S'\subset C_A, C_B$. By (2), there exists faces $F_A$ of $C_A$ and $F_B$ of $C_B$ such that $F \subset F_A \cap F_B$. Clearly faces $F, F_A$ and $F_B$ share the supporting hyperplane $H$; in particular, $F$, $F_A$ and $F_B$ share the same inward normal vector $f$. Therefore, we get $w_0 \in H$ and $B^n(w_0,1/r^{\ref{lem_conelike_1}}_{n \lambda, \ell})\cap H^+ \subset C_A \cap C_B.$
and
$B^n(w_0, 1/r^{\ref{lem_conelike_1}}_{n \lambda, \ell}) \cap H \subset F_A \cap F_B.$
It immediately follows that for every $ w \in B^n(w_0,1/2r^{\ref{lem_conelike_1}}_{n \lambda, \ell})\cap H $, we also have
\begin{align}\label{prop_cone_1.5}
    B^n(w,1/2r^{\ref{lem_conelike_1}}_{n \lambda, \ell})\cap H^+ \subset C_A \cap C_B.
\end{align}
Fix $ w \in B^n(w_0,1/2r^{\ref{lem_conelike_1}}_{n \lambda, \ell})\cap H \subset F_A \cap F_B$. As $C_A \cup C_B \subset B^n(o, \ell)$, it follows that 
\begin{align}\label{prop_cone_2}
    |w| \leq \ell.
\end{align}
Recall that faces $F, F_A$ and $F_B$ share the same inward normal vector $f$. Because $w \in F_A \cap F_B$ and $M(w) \in \partial C_A \cup \partial C_B$, by convexity we deduce that $y_1=M(w)-w$ satisfies
\begin{align}\label{prop_cone_3}
    \frac{\langle y_1, f \rangle}{|y_1||f|} \geq 0.
\end{align}
By \Cref{lem_conelike_final}, together with \eqref{prop_cone_1}, \eqref{prop_cone_2} and \eqref{prop_cone_3}, applied with parameters $n, 2r^{\ref{lem_conelike_1}}_{n \lambda, \ell}, \sigma^{\ref{lem_conelike_3}}_{n,r^{\ref{lem_conelike_1}}_{n \lambda, \ell}}$ (recall $r^{\ref{lem_conelike_1}}_{n, \lambda, \ell}>\ell$), the ball
$X_w= B^n\Big(w+f/(4r^{\ref{lem_conelike_1}}_{n \lambda, \ell}) , k^{\ref{lem_conelike_final}}_{n,2r^{\ref{lem_conelike_1}}_{n \lambda, \ell},\sigma^{\ref{lem_conelike_3}}_{n,r^{\ref{lem_conelike_1}}_{n \lambda, \ell}}}\Big)$ has the following properties:
\begin{enumerate}
    \item $X_w \subset (1-1/(32(r^{\ref{lem_conelike_1}}_{n \lambda, \ell})^2))(B(w,1/2r^{\ref{lem_conelike_1}}_{n \lambda, \ell}) \cap H^+).$
    \item $\mathbb{P}_{x \in X_w} \left(\frac{\langle y_1, x-w \rangle}{|y_1||x-w|} \geq 0\right)\geq 1/2$
    \item For all $x \in X_w$, we have
$\frac{\langle sy_2, x-w \rangle}{|sy_2||x-w|} \geq \sigma^{\ref{lem_conelike_3}}_{n,r^{\ref{lem_conelike_1}}_{n \lambda, \ell}}/4$
    \item For all $x \in X_w$, we have $|x-w| \geq 1/8r^{\ref{lem_conelike_1}}_{n \lambda, \ell}.$
\end{enumerate}
By construction,
$|X_w| =\Big( k^{\ref{lem_conelike_final}}_{n,2r^{\ref{lem_conelike_1}}_{n \lambda, \ell},\sigma^{\ref{lem_conelike_3}}_{n,r^{\ref{lem_conelike_1}}_{n \lambda, \ell}}}\Big)^n|B^n(o,1)| = m^{\ref{prop_conelike_0}}_{n, \lambda, \ell}.$
By the first property of $X_w$, together with \eqref{prop_cone_1.5}, we get
$X_w \subset  (1-1/(32(r^{\ref{lem_conelike_1}}_{n \lambda, \ell})^2)) (C_A \cap C_B) = (1-\varepsilon^{\ref{prop_conelike_0}}_{n, \lambda, \ell})(C_A\cap C_B) $.
The second and third property of $X_w$ exactly give
$\mathbb{P}_{x \in X_w} \left(\frac{\langle y_1, x-w \rangle}{|y_1||x-w|} \geq 0\right)\geq 1/2$
and
$\mathbb{P}_{x \in X_w} \left(\frac{\langle sy_2, x-w \rangle}{|y_2||x-w|} \geq \sigma^{\ref{prop_conelike_0}}_{n, \lambda, \ell}\right) =1.$
The last property of $X_w$ is exactly
$d(w,X_w)\geq 1/(4r^{\ref{prop_conelike_0}}_{n \lambda, \ell}).$
Finally, note that all of these hold for all $ w \in B^n(w_0,1/r^{\ref{prop_conelike_0}}_{n \lambda, \ell})\cap H$, which concludes the proof.
\end{proof}

\subsubsection{Proof of \Cref{lem_conelike_1}}

\begin{proof}[Proof of \Cref{lem_conelike_1}]
Set $r^{\ref{lem_conelike_1}}=2\ell \lambda$. The first two parts follow immediately from \Cref{conelike} (2). For the third part, note that by \Cref{conelike} (1) and (2) we have $B(o, 1/\ell) \subset C_A \subset \lambda S''+z=\lambda S'+(1-\lambda)z .$
After rearranging, we conclude
$B\left(\frac{\lambda-1}{\lambda}z, \frac{1}{\ell \lambda} \right) \subset S'.$
In addition,
$z \in S' \subset C_A \subset B(o, \ell) $.
After rearranging, we conclude
$$S' \subset B(o, \ell) \subset B\left(\frac{\lambda-1}{\lambda}z, \ell + \frac{\lambda-1}{\lambda}|z|\right) \subset B\left(\frac{\lambda-1}{\lambda}z, 2\ell \right).$$
\end{proof}

\subsubsection{Proof of \Cref{lem_conelike_final}}

\begin{proof}[Proof of \Cref{lem_conelike_final}]
Set $k=(4r)^{-1}\sigma$ and $\varepsilon=1/(8r)^2$. As everything is normalized, without loss of generality we can assume $|y_1|=|y_2|=1$.

For the second part, consider the half-space $Y=\{x \colon \langle y_1, x-w \rangle \geq 0\}.$ We need to show that $|X \cap Y|/|X| \geq 1/2.$ Because $X$ is a ball and $Y$ is a half-space, it is enough to show that the center of the ball belongs to the half space. In other words, we need to check
$\langle y_1, w+f/(2r) -w \rangle \geq 0,$
which follows from the hypothesis
$\langle y_1, f \rangle \geq 0.$

For the rest of the proof fix $x \in X= B^n(w+f/(2r),k)$. For the third part, note that we can write $x= w+f/(2r)+\alpha g$ where $g$ is a unit vector and $k \geq \alpha\geq 0$. Thus we have
\begin{align*}
    \begin{split}
        \frac{\langle y_2, x-w \rangle}{|y_2||x-w|} &= \frac{\langle y_2, f/(2r)+\alpha g \rangle}{|y_2||f/(2r)+\alpha g|}
        = \frac{(2r)^{-1}\langle y_2, f \rangle + \alpha \langle y_2, g\rangle }{|y_2||f/(2r)+\alpha g|}
        \geq \frac{(2r)^{-1} \sigma + \alpha \langle y_2, g\rangle }{|y_2||f/(2r)+\alpha g|}\\
        &\geq \frac{(2r)^{-1} \sigma - \alpha }{|y_2||f/(2r)+\alpha g|}
        \geq \frac{(4r)^{-1} \sigma}{|y_2||f/(2r)+\alpha g|}
        \geq \frac{(4r)^{-1} \sigma }{|y_2|(|f/(2r)|+|\alpha g|)}\\
        &\geq \frac{(4r)^{-1} \sigma }{1/(2r)+\alpha}
        \geq \frac{(4r)^{-1} \sigma }{3/(4r)} \geq \sigma/4.
    \end{split}\end{align*}
Here the first inequality follows from the hypothesis $\langle y_2, f \rangle \geq \sigma$. The second inequality follows from the simple fact that for unit vectors $y_2, g$ $\langle y_2, g \rangle \geq -1$. The third inequality follows from the fact that $\alpha \leq k \leq (4r)^{-1}\sigma$. The forth inequality is the triangle inequality. The fifth inequality follows from the fact that $y_2, f, g$ have norm 1. The sixth inequality follows from the fact that $\alpha \leq k \leq (4r)^{-1}$.

For the forth and first parts, we recall that $|x-(w+(f/2r)) | \leq k$, $|w| \leq r$ and $|f|=1$ and apply the triangle inequality.
$$|x-w| \geq |w+f/(2r)-w| - |x-(w+f/(2r))| \geq 1/(2r)-k \geq 1/(4r). $$
Here we used the hypothesis $k \leq 1/(4r)$.
$$|x-(1-\varepsilon)w| \leq |w+f/(2r)-w| +|\varepsilon w| +|x-(w+f/(2r))|   \leq 1/(2r) + \varepsilon r+k \geq 7/(8r) \leq (1-\varepsilon)r. $$
Here we used the hypothesis $k \leq 1/(4r)$ and $\varepsilon \leq 1/(8r^2) \leq 1/8$. 
Finally, we can again write $x=w+f/(2r)+\alpha g$ with $g$ a unit vector and $0\leq \alpha\leq k$, so that we have
$$\langle f, x\rangle=\langle f, w\rangle+\langle f, f/(2r)\rangle+\langle f, \alpha g\rangle= 0+1/(2r)+\alpha \langle f, g\rangle\geq 0.$$
Here we used that $0\leq\alpha\leq k\leq 1/(2r)$ and $\langle f, g\rangle\geq -1$. Hence, we find that $X\subset H^+$.
\end{proof}

\subsubsection{Proof of \Cref{lem_conelike_1.5}}

\begin{proof}[Proof of \Cref{lem_conelike_1.5}]
Let $x$ be the vertex of $S'$ opposite to $F$. Let $v=xu\cap F$ be the intersection of the ray $xu$ with the face $F$. Set $\lambda = |xu|/|xv| \leq 1$. Then it is easy to see that $(1-\lambda) x+ \lambda B^{n}(v,1/r) = B^n(u, \lambda /r) \subset B^n(u,1/r). $

Let $F,F_1, \dots, F_n$ be the faces of $F$ and let $H, H_1, \dots, H_n$ be the supporting hyperplanes, respectively. For each $1 \leq i \leq n$ let $H_i^+$ and $H_i^-$ be the two half spaces determined by $H_i$, such that $H_i^+$ contains $S'$ and $H_i^-$ is disjoint from the interior of $S'$. Then
$S'=H^+ \cap_{i=1}^n H_i^+.$

For the first part, as $B^n(v, 1/r) \cap H^+ \subset H^+$, it is enough to show that for $1\leq i \leq n$, we have $B^n(v, 1/r) \subset H_i^+.$
Assume for the sake of contradiction that there exists $y \in B^n(v, 1/r) \cap H_i^{- \circ}$.  As vertex $x$ belongs to all faces except $F$, we have $x \in F_i \subset H_i \subset H_i^{-}$. Hence, as $H_i^-$ is convex, we have $ (1-\lambda)x+\lambda y \subset H_i^{- \circ}.$
However, by the above discussion, we have
$ (1-\lambda)x+\lambda y \subset B^{n}(u,1/r) \subset S' \subset H_i^+ .$
As $H_i^+$ and $H_i^{-\circ}$ are disjoint, this gives the desired contradiction. Thus, we conclude the first part.

For the second part, on the one hand we have
$B^n(v, 1/r) \cap H = B^n(v, 1/r) \cap H^+\cap H \subset S' \cap H=F.$
On the other hand, $F \subset H$ by definition and $F \subset S' \subset B^n(u,r) \subset B^n(v,2r)$ by hypothesis. For the last inclusion we just used the fact that $v\in F \subset S'  \subset B^n(u,r)$. Thus, we conclude the second part.
\end{proof}

\subsubsection{Proof of \Cref{lem_conelike_3}}

\begin{proof}[Proof of \Cref{lem_conelike_3}]
For a contradiction assume there is a sequence of simplices $S^i$ and unit vectors $v^i$ so that $\max_{1\leq j\leq n} |\langle f^i_j,v^i\rangle|\leq \sigma_i$, where $\sigma_i\to 0$ as $i\to \infty$. By compactness there exists a converging subsequence so that $v^i\to v$ and $S^i\to S'$ (each of the vertices of $S_i$ converging to the corresponding vertices of $S'$). $S'$ has the property that $B(u,1/r)\subset S'\subset \overline{B(u,r)}$ and letting $f_i$ be the inward normal vectors to the faces of $S'$, we have $\langle f_i,v\rangle= 0$ for all $1\leq i\leq n$.

Consider the line $u+\mathbb{R}v$ through $u$. As $B(u,1/r)\subset S'$, this line goes through the interior of $S'$, so intersects the boundary $\partial S'$ exactly twice, in two distinct faces. In particular, this line intersects some face $i$ with normal $f_i$ with  $1\leq i\leq n$.  However, as $\langle f_i,v\rangle= 0$ it follows that this line is contained inside face $i$. However, this line goes through the interior of $S'$,  contradiction.
\end{proof}

\subsubsection{Proof of \Cref{lambdabound}}

\begin{proof}[Proof of \Cref{lambdabound}]
The statement is equivalent to the following statement. There exists $0<\alpha_n<1$ such that the following holds. If $0<t\leq 1/2$ and $\lambda_1, \dots, \lambda_n>0$ and $\lambda_1\dots\lambda_n=1$, then
\begin{align*}
\alpha_n(\lambda_1-1)^2\leq t^{-1}\bigg(\prod_i (t+(1-t)\lambda_i)-1\bigg)+ t^{-2n}\left({\prod_i (t+(1-t)\lambda_i)-1}\right)^2.
\end{align*}
It is easy to check that for fixed $\lambda_1>0$, and conditioned on $\lambda_1\dots \lambda_n=1$, the right hand side is minimised when $\lambda_2=\dots=\lambda_n=\lambda_1^{\frac{1}{1-n}}$. This is because for $a, b>0$ we have $(t+(1-t)a)(t+(1-t)b)\geq (t+(1-t)\sqrt{ab})^2$. Write $\lambda_1=\lambda^{1-n}$ and $\lambda_2=\dots=\lambda_n=\lambda$ for some $\lambda>0$. Then the inequality becomes
\begin{align*}
\alpha_n(\lambda^{1-n}-1)^2\leq t^{-1}\bigg((t+(1-t)\lambda^{1-n})(t+(1-t)\lambda)^{n-1} -1\bigg)+t^{-2n}\bigg((t+(1-t)\lambda^{1-n})(t+(1-t)\lambda)^{n-1} -1\bigg)^2
\end{align*}
We first assume that $0<\lambda \leq 1$ and note that 
\begin{align*}
(t+&(1-t)\lambda^{1-n})(t+(1-t)\lambda)^{n-1} -1\geq \lambda^{(1-t)(1-n)}(t+(1-t)\lambda)^{n-1} -1=(t\lambda^{t-1}+(1-t)\lambda^t)^{n-1} -1\\
&\geq t\lambda^{t-1}+(1-t)\lambda^t -1= t\exp( -\log(\lambda)(1-t) ) + (1-t)\exp(t\log(\lambda))-1\\
&\geq t(1- \log(\lambda)(1-t)+\frac{1}{2}\log^2(\lambda)(1-t)^2) + (1-t)(1+t\log(\lambda))-1=\frac{t}{2} \log^2(\lambda)(1-t)^2 \geq \frac{t}{8} \log^2(\lambda).
\end{align*}
We now assume that $\lambda \geq 1$ and note that
\begin{align*}
(t+&(1-t)\lambda^{1-n})(t+(1-t)\lambda)^{n-1} -1\geq (t+(1-t)\lambda^{1-n})\lambda^{(1-t)(n-1)} -1=t\lambda^{(1-t)(n-1)}+(1-t)\lambda^{-t(n-1)} -1\\
&= t\exp( \log(\lambda)(1-t)(n-1) ) + (1-t)\exp(-t\log(\lambda)(n-1))-1\\
&\geq t\Big(1+ \log(\lambda)(1-t)(n-1)+\frac{1}{2}\log^2(\lambda)(1-t)^2(n-1)^2\Big)+ (1-t)\Big(1-t\log(\lambda)(n-1)\Big)-1\\
&=\frac{t}{2} \log^2(\lambda)(1-t)^2(n-1)^2 \geq \frac{t}{8} \log^2(\lambda).
\end{align*}
In both cases ($0<\lambda\leq 1$ and $\lambda>1$), the first inequality follows from the AM-GM inequality: $px+(1-p)y \geq x^py^{1-p}$ for $0<p<1$ and $0<x,y$. Also, the penultimate inequality follows from the inequalities $\exp(x) \geq 1+x+\frac{x^2}{2}$ and $\exp(-x) \geq 1-x$, for $x \geq 0$. Combining the two cases, for $\lambda>0$ we get that 
\begin{align*}
(t+(1-t)\lambda^{1-n})(t+(1-t)\lambda)^{n-1} -1 \geq \frac{t}{8} \log^2(\lambda).
\end{align*}
Therefore, for $|\lambda-1|\leq 0.25$, using the simple inequality $|\log(\lambda)| \geq \frac{|\lambda-1|}{2}$, we deduce that
\begin{align}\label{l1}
(t+(1-t)\lambda^{1-n})(t+(1-t)\lambda)^{n-1} -1 \geq \frac{t}{2^5} |\lambda -1|^2.
\end{align}
Moreover, for $|\lambda-1| \geq 0.25$, using the inequality $|\log(\lambda)| \geq 1/8$, we deduce that
$
(t+(1-t)\lambda^{1-n})(t+(1-t)\lambda)^{n-1} -1 \geq \frac{t}{2^{9}}.$
The last inequality implies that for $t/2<\lambda \leq 0.75$, we have
\begin{align}\label{l2}
\begin{split}
  t^{-2n}\Big((t+(1-t)\lambda^{1-n})(t+(1-t)\lambda)^{n-1} -1\Big)^2 &\geq \frac{t^{2-2n}}{2^{18}}\geq   \frac{\lambda^{2-2n}}{2^{16+2n}}\geq   \frac{(\lambda^{1-n}-1)^2}{2^{16+2n}}.
\end{split}
\end{align}
It also implies that for $1.25 \leq \lambda$, we have
\begin{align}\label{l3}
\begin{split}
  t^{-2n}\Big((t+(1-t)\lambda^{1-n})(t+(1-t)\lambda)^{n-1} -1\Big)^2 &\geq \frac{t^{2-2n}}{2^{18}}\geq   \frac{1}{2^{18}}\geq   \frac{(\lambda^{1-n}-1)^2}{2^{18}}.
\end{split}
\end{align}
For $0<\lambda \leq t/2$, we have the simple bound 
\begin{align*}
\begin{split}
 (t+(1-t)\lambda^{1-n})(t+(1-t)\lambda)^{n-1} -1 &\geq (1-t) \lambda^{1-n} t^{n-1} -1 \geq 3^{-1}\lambda^{1-n} t^{n-1}, 
\end{split}
\end{align*}
where the last inequality follows from $\lambda^{1-n} t^{n-1} \geq 2^{n-1}$ and $0<t \leq 1/2$. Therefore, for $0<\lambda \leq t/2$ we infer
\begin{align}\label{l4}
\begin{split}
 t^{-2n}\Big((t+(1-t)\lambda^{1-n})(t+(1-t)\lambda)^{n-1} -1\Big)^2 
 &\geq t^{-2n} 3^{-2}\lambda^{2-2n} t^{2n-2}\geq 3^{-2}\lambda^{2-2n} \geq 3^{-2}(\lambda^{1-n}-1)^2.
\end{split}
\end{align}
where the last inequality follows from $\lambda \leq 1$.

Combining \eqref{l1}, \eqref{l2}, \eqref{l3} and \eqref{l4}, we conclude 
\begin{align*}
\frac{(\lambda^{1-n}-1)^2}{2^{18+2n}}\leq t^{-1}\bigg((t+(1-t)\lambda^{1-n})(t+(1-t)\lambda)^{n-1} -1\bigg)+t^{-2n}\bigg((t+(1-t)\lambda^{1-n})(t+(1-t)\lambda)^{n-1} -1\bigg)^2.
\end{align*}
\end{proof}

\subsubsection{Proof of \Cref{lem_filling}}

\begin{proof}[Proof of \Cref{lem_filling}]
Choose maximal $1\geq \mu>0$ such that $(\ell^{-2}+1)^{-1}(1+\mu) \leq (t\varepsilon/4)^{-1}(t(1+\mu)\varepsilon/4-\mu)$.
By hypothesis, we have
$K-x\subset A\subset C_A \subset (1+\mu)K-x \text{ and }
K-y \subset B\subset C_B \subset (1+\mu) K -y.$

Therefore, it is enough to show that
$t(K-x)+(1-t)(K-y) \supset t(1-\varepsilon/4)((1+\mu)K-x)+ (1-t)((1+\mu) K -y).$
After rearranging, this is equivalent to
$K -tx-(1-t)y\supset (1-t\varepsilon/4)(1+\mu)K-t(1-\varepsilon/4)x-(1-t)y.$
After further rearranging, this is equivalent to
$K \supset (1+\mu-t(1+\mu)\varepsilon/4)K+(t\varepsilon/4) \ x.$
Therefore, it is enough to show
$$x \in (t\varepsilon/4)^{-1}(t(1+\mu)\varepsilon/4-\mu)K $$

By hypothesis, we know
$K \subset (1+\mu) K$
which implies that 
$o \in K$ (assuming wlog $K$ is compact).
By hypothesis, we also know
$K-x \subset B(o,\ell)$.
Combining the last two inclusions, we get
$-x \in B(o, \ell) \text{ i.e., } x \in B(o, \ell).$
Finally, by hypothesis we have
$(1+\mu)K-x \supset B(o, \ell^{-1}). $

Combining the last two inclusions and rearranging, we get
$$ x\in (\ell^{-2}+1)^{-1}(1+\mu)K.$$

By the choice of parameters, we have $(\ell^{-2}+1)^{-1}(1+\mu) \leq (t\varepsilon/4)^{-1}(t(1+\mu)\varepsilon/4-\mu)$, from which the conclusion follows.
\end{proof}

\subsubsection{Proof of \Cref{lem_ellipticregularity}}

\begin{proof}[Proof of \Cref{lem_ellipticregularity}]
We first observe that, by Caffarelli's regularity theory \cite{Caff92int,Caff92bdry}, the function $\varphi$ is a strictly convex Alexandrov solution 
of $\det D^2\varphi=1$. Also, thanks to \eqref{eq:normalized}, the modulus of strict convexity depends only on $R$ and the dimension.
Hence, we can apply the interior regularity theory for Alexandrov solutions (see for instance \cite[Theorem 4.42]{FigMAbook}) to deduce that, for every $\theta ,\alpha \in (0,1)$, $D^2\varphi$ is uniformly $\alpha$-H\"older continuous inside $(1-\theta)C_A$. More precisely, there exists a constant $\hat C_{n,R,\theta,\alpha}>0$ such that 
\begin{equation}
\label{eq:C2a}
\|D^2\varphi\|_{C^{0,\alpha}((1-\theta)C_A)} := \|D^2\varphi\|_{L^\infty((1-\theta)C_A)}+\sup_{x, y \in (1-\theta)C_A}\frac{|D^2\varphi(x)-D^2\varphi(y)|}{|x-y|^\alpha}\leq \hat C_{n,R,\theta,\alpha}
\end{equation}
(here the choice of the norm for $D^2\varphi(x)$ is irrelevant, since all norms are equivalent up to dimensional constants).

Now, given any affine function $\ell(x):=b\cdot x+c$ ($b \in \mathbb R^n$, $c \in \mathbb R$),  consider the second-order polynomial $p_\ell(x):=\frac{|x|^2}{2}+\ell(x)$. Since $\det D^2p_\ell=1$, applying \cite[Lemma A.1]{FigMAbook} we  write
\begin{align*}
0=\det D^2\varphi -\det D^2p_\ell&=\int_0^1\frac{d}{dt}\det\big(tD^2\varphi+(1-t)D^2p_\ell\big)\,dt\\
&=\sum_{i,j=1}^n\biggl(\int_0^1{\rm cof}\big(tD^2\varphi+(1-t){\rm Id}\big)\,dt\biggr)_{ij}\partial_{ij}(\varphi-p_\ell),
\end{align*}
where, given a symmetric matrix $A$, ${\rm cof}(A)$ denotes its cofactor matrix. In other words, if we define the functions
$$
a_{ij}(x):=\biggl(\int_0^1{\rm cof}\big(tD^2\varphi(x)+(1-t){\rm Id}\big)\,dt\biggr)_{ij}\qquad i,j\in \{1,\ldots,n\}
$$
and $\psi_\ell:=\varphi-p_\ell$,  then $\psi_\ell$ solves the equation
$$
\sum_{i,j=1}^na_{ij}\partial_{ij}\psi_\ell=0.
$$
Note that, thanks to \eqref{eq:C2a}, the matrices $(a_{ij}(x))_{i,j=1}^n$ are uniformly positive definite and H\"older continuous. Hence, recalling  \eqref{eq:normalized}, it follows from classical elliptic regularity (see for instance \cite[Corollary 6.3 and Theorem 9.20]{GT98}) and a covering argument that
\begin{equation}
\label{eq:elliptic reg}
\|D^2\psi_\ell\|_{L^\infty((1-2\theta)C_A)} \leq C'_{n,R,\theta} \|\psi_\ell\|_{L^1((1-\theta)C_A)},
\end{equation}
where $C'_{n,R,\theta}$ depends on $n$, $R$, and $\theta$ only.

Now, set $\psi(x):=\varphi(x)-\frac{|x|^2}{2}$ and fix $\bar \ell(x)=\bar b\cdot x+\bar c$ with
$$
\bar b:=\frac{1}{|(1-\theta)C_A|}\int_{(1-\theta)C_A}\nabla \psi(x)\,dx,\qquad
\bar c:=\frac{1}{|(1-\theta)C_A|}\int_{(1-\theta)C_A}(\psi(x)-\bar b\cdot x)\,dx.
$$
Then, by applying twice the 1-Poincar\'e inequality (see \cite[Equation (7.45)]{GT98} with $p=1$) and recalling \eqref{eq:normalized}, we have
\begin{equation}
\label{eq:poincare}
\|\psi_{\bar \ell}\|_{L^1((1-\theta)C_A)} \leq 2^n R^{2n-1} \|\nabla \psi_{\bar \ell} \|_{L^1((1-\theta)C_A)} \leq 4^n R^{4n-2}\|D^2 \psi_{\bar \ell}\|_{L^1((1-\theta)C_A)}.
\end{equation}
Noticing that $D^2\psi_{\bar \ell}=D^2\psi$, combining \eqref{eq:elliptic reg} (with $\ell=\bar \ell$) and \eqref{eq:poincare} we conclude that
$$
\|D^2\psi\|_{L^\infty((1-2\theta)C_A)} \leq 4^nR^{4n-2}C'_{n,R,\theta}\|D^2 \psi\|_{L^1((1-\theta)C_A)}.
$$
Choosing $\theta=\varepsilon/2$, this proves the desired estimate with $C_{n,R,\varepsilon}=4^n R^{4n-2} C'_{n,R,\varepsilon/2}$.
\end{proof}

\section{Putting it all together: Proof of \Cref{main_thm_5}}

\begin{proof}[Proof of \Cref{main_thm_5}]
Consider any $n,t,\ell,$ and $\lambda$.
Choose $g_{n,t}=d_{n,t}:=t^{2n-1}/4$. Choose $\theta=\theta_{n,\ell/2}^{\ref{probabilisticlem}}$ as given by \Cref{probabilisticlem}. Choose $\varepsilon=\frac12$. Choose $\mu:=\mu^{\ref{mainpropmass}}_{n,t,\varepsilon,\ell}$ as given by \Cref{mainpropmass}. Choose $c_{n,\ell,\lambda}:=\frac{c^{\ref{sdvsdistanceint}}_{n,\ell}c^{\ref{maincormass}}_{n,\varepsilon, \lambda, \ell, \theta}}{c^{\ref{probabilisticlem}}_{n,\ell}}\theta^{-n}+2$, where $c^{\ref{sdvsdistanceint}}_{n,\ell},c^{\ref{maincormass}}_{n,\varepsilon, \lambda, \ell, \theta}, $and $c^{\ref{probabilisticlem}}_{n,\ell}$ are the constants from \Cref{sdvsdistanceint}, \Cref{maincormass} and \Cref{probabilisticlem} respectively.

By \Cref{transportinitialreduction}, we may assume that $A,B$ are $(\gamma,\ell,\lambda,\mu)$ conelike with $\mu$ sufficiently small in terms of $n,t,\ell$ and $\lambda$.

By \Cref{maincormass}, we find that for any affine transformation $Q:\mathbb{R}^n\to\mathbb{R}^n$, if $||Q||_{op},||Q^{-1}||_{op}\leq \theta$ and  if $T_Q:Q(C_A)\to Q(C_B)$ is the optimal transport map from $Q(C_A)$ to $Q(C_B)$, then
\begin{align*}\int_{x\in\partial C_A}\max\left\{ \left\langle Q(x)-T_Q(Q(x)),\frac{Q(x)-Q(o')}{||Q(x)-Q(o')||_2}\right\rangle, 0 \right\} dx\leq  c^{\ref{maincormass}}_{n,\varepsilon, \lambda, \ell, \theta}\sqrt{\frac{\delta+\gamma}{t}}|Q(A)|,
\end{align*}
for all $o'\in (1-\varepsilon)C_A$.

Let $Q\sim\mathcal{Q}_\theta$ be random scaling, choose $o'$ uniformly random from $B(o,1/2\ell)$. Note that $T_Q(Q(C_A))\subset Q(C_B)$ and $B(o,1/2\ell)\subset \frac12 C_A, \frac12 C_B\subset B(o,2\ell)$. Hence, by \Cref{probabilisticlem}, we have
\begin{align*}
    \mathbb{E}_{Q,o'}\left[ \max\left\{ \left\langle Q(x)-T_Q(Q(x)),\frac{Q(x)-Q(o')}{||Q(x)-Q(o')||_2}\right\rangle, 0 \right\} \right]&\geq  c^{\ref{probabilisticlem}}_{n, 2\ell} d(x, C_B).
\end{align*}

Since we have $||Q||_{op},||Q^{-1}||_{op}\leq \theta$ for every random scaling $Q\sim \mathcal{Q}_\theta$, and $B(o,1/2\ell)\subset (1-\varepsilon)C_A, (1-\varepsilon)C_B$, we can combine these two to find:
\begin{align*}
c^{\ref{probabilisticlem}}_{n,2\ell}\int_{x\in \partial C_A}  &d(x, C_B)dx\leq \int_{x\in \partial C_A}  \mathbb{E}_{Q,o'}\left[ \max\left\{ \left\langle Q(x)-T_Q(Q(x)),\frac{Q(x)-Q(o')}{||Q(x)-Q(o')||_2}\right\rangle, 0 \right\} \right]dx\\
&= \mathbb{E}_{Q,o'}\left[\int_{x\in \partial C_A}   \max\left\{ \left\langle Q(x)-T_Q(Q(x)),\frac{Q(x)-Q(o')}{||Q(x)-Q(o')||_2}\right\rangle, 0 \right\} dx\right]\\
&\leq \mathbb{E}_{Q,o'}\left[c^{\ref{maincormass}}_{n,\varepsilon, \lambda, \ell, \theta}\sqrt{\frac{\delta+\gamma}{t}}|Q(A)|\right]\leq c^{\ref{maincormass}}_{n,\varepsilon, \lambda, \ell, \theta}\sqrt{\frac{\delta+\gamma}{t}}\mathbb{E}_{Q}\left[|Q(A)|\right]\leq c^{\ref{maincormass}}_{n,\varepsilon, \lambda, \ell, \theta}\sqrt{\frac{\delta+\gamma}{t}}\theta^{n}|A|,
\end{align*}
where the final inequality follows as every $Q\sim\mathcal{Q}_\theta$ has determinant at most $\theta^{n}$.
Applying \Cref{sdvsdistanceint} we find
$$|C_A\triangle C_B|\leq c^{\ref{sdvsdistanceint}}_{n,\ell}\int_{x\in \partial C_A}  d(x, C_B)dx\leq \frac{c^{\ref{sdvsdistanceint}}_{n,\ell}c^{\ref{maincormass}}_{n,\varepsilon, \lambda, \ell, \theta}}{c^{\ref{probabilisticlem}}_{n,2\ell}}\sqrt{\frac{\delta+\gamma}{t}}\theta^{n}|A|.$$
We conclude recalling the definition of $C_A,C_B$:
$$|A\triangle B|\leq |C_A\triangle C_B|+|C_A\setminus A|+|C_B\setminus B|\leq \frac{c^{\ref{sdvsdistanceint}}_{n,\ell}c^{\ref{maincormass}}_{n,\varepsilon, \lambda, \ell, \theta}}{c^{\ref{probabilisticlem}}_{n,2\ell}}\sqrt{\frac{\delta+\gamma}{t}}\theta^{n}|A|+2\gamma|A|\leq c_{n,\ell,\lambda}\sqrt{\frac{\delta+\gamma}{t}}|A|.$$
This concludes the proof of the theorem.
\end{proof}

\bibliographystyle{alpha}
\bibliography{references}

\end{document}